\documentclass[a4paper,11pt]{article}
\usepackage{amsmath, amsfonts, amscd, amssymb, amsthm, enumerate, xypic}

\def\bfB{\mathbf{B}}
\def\bfC{\mathbf{C}}
\def\bfD{\mathbf{D}}

\newcommand{\Diag}{\operatorname{Diag}}
\newcommand{\Mat}{\operatorname{M}}

\newcommand{\Mata}{\operatorname{A}}

\newcommand{\id}{\operatorname{id}}
\newcommand{\GL}{\operatorname{GL}}

\newcommand{\Ker}{\operatorname{Ker}}

\newcommand{\Vect}{\operatorname{span}}
\newcommand{\im}{\operatorname{Im}}
\newcommand{\urk}{\operatorname{urk}}
\newcommand{\mrk}{\operatorname{mrk}}

\newcommand{\tr}{\operatorname{tr}}

\newcommand{\rk}{\operatorname{rk}}

\renewcommand{\setminus}{\smallsetminus}

\def\F{\mathbb{F}}
\def\K{\mathbb{K}}
\def\R{\mathbb{R}}
\def\C{\mathbb{C}}

\def\N{\mathbb{N}}

\renewcommand{\L}{\mathbb{L}}

\def\calA{\mathcal{A}}
\def\calB{\mathcal{B}}
\def\calC{\mathcal{C}}
\def\calD{\mathcal{D}}

\def\calH{\mathcal{H}}

\def\calL{\mathcal{L}}
\def\calM{\mathcal{M}}
\def\calN{\mathcal{N}}

\def\calP{\mathcal{P}}

\def\calR{\mathcal{R}}
\def\calS{\mathcal{S}}
\def\calT{\mathcal{T}}
\def\calU{\mathcal{U}}
\def\calV{\mathcal{V}}
\def\calW{\mathcal{W}}

\def\lcro{\mathopen{[\![}}
\def\rcro{\mathclose{]\!]}}

\theoremstyle{definition}
\newtheorem{Def}{Definition}[section]

\theoremstyle{plain}
\newtheorem{theo}{Theorem}[section]
\newtheorem{prop}[theo]{Proposition}
\newtheorem{cor}[theo]{Corollary}
\newtheorem{lemme}[theo]{Lemma}
\newtheorem{claim}{Claim}

\theoremstyle{plain}

\theoremstyle{remark}
\newtheorem{Rems}{Remarks}[section]
\newtheorem{Rem}[Rems]{Remark}
\newtheorem{ex}[Rems]{Example}

\title{Local linear dependence seen through duality I}
\author{Cl\'ement de Seguins Pazzis\footnote{Universit\'e de Versailles Saint-Quentin-en-Yvelines, Laboratoire de Math\'ematiques
de Versailles, 45 avenue des Etats-Unis, 78035 Versailles cedex, France}
\footnote{e-mail address: dsp.prof@gmail.com}}

\begin{document}

\thispagestyle{plain}

\maketitle


\begin{abstract}
A vector space $\calS$ of linear operators between vector spaces $U$ and $V$ is called locally linearly dependent
(in abbreviated form: LLD) when every vector $x \in U$ is annihilated by a non-zero operator in $\calS$.
A duality argument bridges the theory of LLD spaces to the one of vector spaces of non-injective operators.
This new insight yields a unified approach to rediscover basic LLD theorems and obtain many additional ones
thanks to the power of formal matrix computations.

In this article, we focus on the minimal rank for a non-zero operator in an LLD space.
Among other things, we reprove the Bre\v sar-\v Semrl theorem, which states that an $n$-dimensional LLD operator space always contains a
non-zero operator with rank less than $n$, and we improve the Meshulam-\v Semrl theorem that
examines the case when no non-zero operator has rank less than $n-1$.

We also tackle the minimal rank problem for a non-zero operator in an $n$-dimensional operator space that is not algebraically reflexive.
A theorem of Meshulam and \v Semrl states that, for all fields with cardinality large enough,
a non-reflexive operator space with dimension $n$ must contain a non-zero operator with rank at most $2n-2$.
We show that there are infinitely many integers $n$ for which this bound is optimal
for general infinite fields. Moreover, under mild cardinality assumptions,
we obtain a complete classification of the non-reflexive $n$-dimensional operator spaces in which no non-zero operator has rank less than $2n-2$.
This classification involves a new algebraic structure called left-division-bilinearizable (in abbreviated form: LDB)
division algebras, which generalize a situation that is encountered with quaternions and octonions and
whose systematic study occupies a large part of the present article.
\end{abstract}



\section{Introduction}

\subsection{Local linear dependence}

Let $U$ and $V$ be vector spaces over a commutative field $\K$.
Linear operators $f_1,\dots,f_n$ from $U$ to $V$ are \textbf{locally linearly dependent (LLD)}
when, for every $x \in U$, the vectors $f_1(x),\dots,f_n(x)$ are linearly dependent in $V$.
There are two obvious situations in which the family $(f_1,\dots,f_n)$ is LLD, in which case we say that
$f_1,\dots,f_n$ are \textbf{trivially LLD}:
\begin{itemize}
\item $f_1,\dots,f_n$ are linearly dependent in the vector space $\calL(U,V)$ of all linear operators from $U$ to $V$.
\item There exists a finite-dimensional subspace $W$ of $V$ such that $\dim W<n$ and $f_i(x) \in W$ for every
$i \in \lcro 1,n\rcro$ and every $x \in U$.
\end{itemize}
Let $\calS$ be a finite-dimensional linear subspace of $\calL(U,V)$. We say that $\calS$ is locally
linearly dependent (LLD)
when every vector $x \in U$ is annihilated by some non-zero operator $f \in \calS$.
Given a positive integer $c$, we say that $\calS$ is $c$-locally linearly dependent ($c$-LLD)
when, for every vector $x \in U$, the linear subspace $\{f \in \calS : f(x)=0\}$ has dimension greater
than or equal to $c$.
Obviously, the operators $f_1,\dots,f_n$ are locally linearly dependent if and only if
they are linearly dependent or $\Vect(f_1,\dots,f_n)$ is locally linearly dependent.

\vskip 3mm
The systematic study of LLD systems of operators is a surprisingly fresh research area in operator theory and linear algebra.
Although some basic results of the theory, such as the fact that an LLD pair $(f,g)$ of operators is always trivial,
have widespread usage in operator theory, it is only quite recently that theorems for more general LLD systems
have started to blossom. We regard the topic as worth investigating for itself,
in particular because it is a natural generalization of Kaplansky's theorem stating that, if an
endomorphism $u$ of a vector space is such that
$(\id,u,u^2,\dots,u^n)$ is LLD, then $\id,u,\dots,u^n$ are linearly dependent.
However, LLD operators also have interesting applications to various problems in mathematics,
and it is those applications that have shaped the topic as it is today.
Amitsur \cite{Amitsur} showed that finding non-zero operators of small rank in LLD operator spaces
yields results on rings satisfying generalized polynomial identities (in that, the operators are actually homomorphisms of abelian groups,
with the target group enriched with a structure of left vector space over a skew field).
In particular, he proved that an $n$-dimensional LLD operator space
always contains a non-zero operator with rank less than $\dbinom{n+1}{2}-1$.
More recently, Bre\v sar and \v Semrl \cite{BresarSemrl} have classified the $3$-dimensional LLD operator spaces
over infinite fields in order to classify commuting pairs $(d,g)$ of continuous derivations of a Banach
algebra $\calA$ such that $dg(x)$ is quasi-nilpotent for all $x\in \calA$.
In their work, they also improved Amitsur's upper bound for the minimal rank of a non-zero operator in an LLD space
by showing that it is actually less than $n$ for an infinite field.

The theory of LLD spaces is also interesting for its connection with the currently fashionable topic of algebraic reflexivity.
Recall that a vector space $\calS$ of linear operators from $U$ to $V$ is called \textbf{algebraically reflexive}
when every linear operator $f  : U \rightarrow V$ which satisfies $\forall x \in U, \; f(x) \in \calS x$ belongs to $\calS$.
If, on the contrary, one can find an operator $f : U \rightarrow V$ which satisfies $\forall x \in U, \; f(x) \in \calS x$
but does not belong to $\calS$, then one sees that the operator space $\calS \oplus \K f$ is LLD, and $\calS$
is a (linear) hyperplane of it. If one is able to give an upper bound for the minimal rank of a non-zero element in
a hyperplane of an LLD space, then one recovers a sufficient condition for algebraic reflexivity.
Before the present article, the best known sufficient condition for algebraic reflexivity based upon the minimal rank of non-zero operators
had been obtained by Meshulam and \v Semrl for fields with cardinality large enough thanks to this method \cite{MeshulamSemrlLAA}.

\vskip 2mm
In the current state of the topic, the various problems that have been investigated are the following ones:
\begin{enumerate}[(1)]
\item \emph{The minimal rank problem for LLD spaces:}
give an upper bound for the minimal rank of a non-zero operator in an $n$-dimensional LLD space.

\item  \emph{The minimal rank problem for hyperplanes of LLD spaces.}
Studying this yields sufficient conditions for reflexivity based upon the minimal rank.

\item \emph{The classification problem:} classify minimal LLD operator spaces.
Prior to this paper, such classifications were known only for $2$-dimensional and $3$-dimensional LLD spaces,
and for $n$-dimensional minimal LLD spaces with an essential range (see Definition \ref{defreduced})
of dimension $\dbinom{n}{2}$ and, in each case, provided that the underlying field is large enough \cite{ChebotarSemrl}.

\item \emph{The maximal rank problem for minimal LLD spaces:}
give an upper bound for the maximal rank in a minimal LLD space of dimension $n$ (note that such an operator space
contains only operators of finite rank, see Proposition \ref{minimalisfinitedimensional}).
\end{enumerate}

\vskip 2mm
Until now, the theory of LLD operators has been developed as an independent subject, with no reference to prior existing
theories save for very basic tools of linear algebra. Our main point is that, with a simple duality argument, we can translate
all the above problems into questions on spaces of non-injective linear operators, and, in the finite-dimensional setting,
into questions on spaces of matrices with bounded rank, a theory for which
many powerful computational tools and classification theorems are now available. Although this duality idea has already surfaced
in earlier works \cite{BresarSemrl,MeshulamSemrlLAA}, we believe that its potential has been dramatically underestimated.
In some instances, this has led the above cited authors to obtain non-optimal results, and in some other ones,
to rediscover previously known results with far longer proofs without realizing it (compare \cite{ChebotarSemrl}
with \cite{AtkinsonPrim}).

Thus, this article and its sequel have two aims:
\begin{itemize}
\item Reboot the theory of LLD operator spaces by establishing the duality argument as the prime method.
We will show how this argument helps one to prove most if not all the known results of the theory in a way that is both efficient, intuitive and elegant.

\item Use the duality argument and the power of matrix computations to advance the theory of LLD operators in all directions.
With this method, we will obtain improved results for \emph{all} the above four problems.
\end{itemize}

In this first article, we shall focus on minimal rank problems, i.e. problems (1) and (2) above.
The next article will be devoted to advances in the classification problem and in the maximal rank problem.

\vskip 2mm
In the next section, we lay out the duality argument.

\subsection{The duality argument}\label{dualityargument}

Let $\calS$ be an $n$-dimensional linear subspace of $\calL(U,V)$.
The adjoint map of the natural embedding $\calS \hookrightarrow \calL(U,V)$
is the linear map
$$x \in U \longmapsto  \bigl[f \mapsto f(x)\bigr] \in \calL(\calS,V),$$
the image of which we denote by
$$\widehat{\calS}:=\bigl\{f \mapsto f(x) \mid x \in U\bigr\} \subset \calL(\calS,V).$$
Then, $\calS$ is LLD if and only if
$\widehat{\calS}$ is \textbf{defective}, i.e.\ no operator in $\widehat{\calS}$ is injective,
i.e.\
$$\forall \varphi \in \widehat{\calS}, \quad \rk \varphi <n.$$
Moreover, $\calS$ is $c$-LLD if and only if $\widehat{\calS}$ is $c$-\textbf{defective},
i.e.\ the kernel of every operator in $\widehat{\calS}$ has dimension greater than or equal to $c$, that is
$$\forall \varphi \in \widehat{\calS}, \quad \rk \varphi \leq n-c.$$

Assume furthermore that $V$ has finite dimension $m$. By choosing respective bases of $\calS$ and $V$,
we may represent $\widehat{\calS}$ by a linear subspace $\mathcal{M}$ of $\Mat_{m,n}(\K)$
(the space of $m \times n$ matrices with entries in $\K$).
For $V$ to be LLD (respectively, $c$-LLD), it is necessary and sufficient that $\rk M \leq n-1$
(respectively, $\rk M \leq n-c$) for all $M \in \calM$.
Thus, studying LLD operator spaces essentially amounts to studying linear subspaces of matrices with rank less than the number of columns.
The theory of such matrix spaces has made good progress in the last sixty years,
see e.g.\ \cite{AtkinsonPrim,AtkLloydPrim,Flanders}, and although some typical problems on LLD spaces
do not have a well-studied counterpart in the theory of matrix spaces with bounded rank, the use of computational tools
from the latter will help us advance them substantially.

\vskip 2mm
To give a simple though illuminating illustration of the power of the duality method,
let us solve the classification problem for $n=2$. Assume that $\calS$ is a $2$-dimensional LLD subspace of $\calL(U,V)$.
Then, $\widehat{\calS}$ is a linear subspace of $\calL(\calS,V)$ in which all the non-zero operators have rank $1$.
By a classical result that is generally attributed to Isaiah Schur,
it follows that either there is a linear hyperplane of $\calS$ on the whole of which all the elements of $\widehat{\calS}$ vanish,
or there is a $1$-dimensional subspace of $V$ which contains the range of every element of $\widehat{\calS}$.
However, the former case cannot hold as it would mean that some non-zero operator $f \in \calS$ satisfies $f(x)=0$ for all $x \in U$.
This yields a $1$-dimensional subspace $D$ of $V$ which contains $f(x)$ for every $f \in \calS$ and every $x \in U$.
Thus, every LLD pair of operators is trivially LLD.

\subsection{Additional definition and notation}

Following the French convention, we use $\N$ to denote the set of non-negative integers.

Given non-negative integers $m$ and $n$, we denote by $\Mat_{m,n}(\K)$ the vector space of matrices with $m$ rows, $n$ columns and entries in $\K$.
We denote by $\Mat_n(\K)$ the algebra of square matrices with $n$ rows and entries in $\K$, by $I_n$ its unit element, and by
$\GL_n(\K)$ its group of invertible elements. The transpose of a matrix $M$ is denoted by $M^T$.
One denotes by $\Mata_n(\K)$ the subspace of alternating matrices of $\Mat_n(\K)$
(i.e.\ skew-symmetric matrices with all diagonal entries zero).

\begin{Def}
Let $\calS$ be a linear subspace of $\calL(U,V)$.
We say that $\calS$ has \textbf{finite minimal rank} when it contains a non-zero operator of finite rank,
in which case we define the \textbf{minimal rank} of $\calS$, denoted by $\mrk(\calS)$,
as the smallest rank among the \emph{non-zero} finite rank operators in $\calS$.

Assume now that $U$ is finite-dimensional.
The maximal rank of an operator in $\calS$ is called the \textbf{upper-rank} of $\calS$
and denoted by $\urk(\calS)$.

For linear subspaces of matrices, we define the same notation by identifying $\Mat_{m,n}(\K)$
with $\calL(\K^n,\K^m)$ in the usual way.
\end{Def}

In later parts of the article, the following notions shall be useful:

\begin{Def}\label{defreduced}
Let $\calS$ be a linear subspace of $\calL(U,V)$. \\
The \textbf{kernel} of $\calS$ is defined as $\underset{f \in \calS}{\bigcap} \Ker f$, i.e.\ the set of vectors of $U$
at which all the operators in $\calS$ vanish. \\
The \textbf{essential range} of $\calS$ is defined as $\underset{f \in \calS}{\sum} \im f$, i.e.\ the linear subspace of $V$
spanned by the ranges of the operators in $\calS$ (note that the union of those ranges is generally not a linear subspace of $V$). \\
We say that $\calS$ is \textbf{reduced} when its kernel is $\{0\}$ and its essential range is $V$.
\end{Def}

Studying $c$-LLD operator spaces amounts to studying reduced ones.
Indeed, let $\calS$ be a linear subspace of $\calL(U,V)$ with kernel $U_0$
and essential range $V_0$. For every operator $f \in \calS$, the inclusions $U_0 \subset \Ker f$ and $\im f \subset V_0$ show that $f$ induces a
linear mapping
$$\overline{f} : [x] \in U/U_0 \longmapsto f(x) \in V_0.$$
Then,
$$\overline{\calS}:=\Bigl\{\overline{f} \;\mid\; f \in \calS\Bigr\}$$
is a reduced linear subspace of $\calL(U/U_0,V_0)$ and $f \mapsto \overline{f}$ is a rank-preserving isomorphism from $\calS$ to $\overline{\calS}$.
One notes that $\calS$ is $c$-LLD if and only if $\overline{\calS}$ is $c$-LLD.
We say that $\overline{\calS}$ is the \textbf{reduced space attached to $\calS$.}

\begin{Def}
Let $\calS$ and $\calS'$ be two vector spaces of linear operators, from $U$ to $V$ and from $U'$ to $V'$, respectively. \\
We say that $\calS$ is \textbf{equivalent to} $\calS'$, and we write $\calS \sim \calS'$, when there are isomorphisms
$F : U \overset{\simeq}{\rightarrow} U'$ and $G : V' \overset{\simeq}{\rightarrow} V$ such that
$\calS=\{G \circ g \circ F \mid g \in \calS'\}$. In that case, we note that $H : f \mapsto G^{-1} \circ f \circ F^{-1}$
is an isomorphism from $\calS$ to $\calS'$, so that
$$\forall f \in \calS, \quad f=G \circ H(f) \circ F.$$
In the special case when $U=V$ and $U'=V'$, we say that $\calS$ is \textbf{similar to} $\calS'$ when,
in the above condition, one can take $G=F^{-1}$.
\end{Def}

When $U$, $V$, $U'$ and $V'$ are all finite-dimensional, we note that $\calS$ is equivalent to $\calS'$ if and only if
both operator spaces are represented by the same set of matrices in different choices of bases of the source and target spaces.

\vskip 2mm
We adopt similar definitions for spaces of matrices.
In particular, we say that two vector spaces $\calM$ and $\calM'$ of $m \times n$ matrices are \textbf{equivalent}
when there are matrices $P \in \GL_m(\K)$ and $Q \in \GL_n(\K)$ such that $\calM=P\calM' Q$.
This means that $\calM$ and $\calM'$ represent the same operator space in different choices of bases.
A vector space $\calM$ of $m \times n$ matrices is called \textbf{reduced} when the space of linear operators
from $\K^n$ to $\K^m$ which it represents in the canonical bases is reduced, which means that
no non-zero vector $X \in \K^n$ satisfies $\calM X=\{0\}$, and no non-zero vector $Y \in \K^m$ satisfies $Y^T \calM=\{0\}$;
in other words, $\calM$ is inequivalent to a space of matrices with last column zero, and
$\calM$ is inequivalent to a space of matrices with last row zero.

\subsection{Structure of the article}

Section \ref{basiclemmassection} regroups some basic results on spaces of matrices with rank bounded above.
In particular, the Flanders-Atkinson lemma is restated there and given a short new proof.

Section \ref{smallranksection} is devoted to minimal rank theorems for LLD operator spaces. There, we reframe the proof
of the Bre\v sar-\v Semrl theorem (and its extension by Meshulam and \v Semrl) as a
straightforward corollary to the Flanders-Atkinson lemma (for fields with large cardinality).
The Meshulam-\v Semrl theorem on the critical rank is also given an improved treatment there,
and our new method helps us slightly relax the cardinality assumption on the field.
At the end of the section, we show that all the operators in a minimal $c$-LLD space
have finite rank: an important application is that reduced minimal LLD spaces have finite-dimensional source and target spaces.

The rest - and the largest part - of the article deals with the minimal rank problem in a hyperplane of an LLD space and its connection with
non-reflexive operator spaces. In \cite{MeshulamSemrlLAA}, Meshulam and \v Semrl showed that in such a non-reflexive $n$-dimensional space of operators,
there always exists a non-zero operator with rank at most $2n-2$ (provided that the underlying field has more than $n$ elements),
and even one with rank at most $n$ if the underlying field is algebraically closed.
Meshulam and \v Semrl conjectured that the latter result held for any large enough field. This however
fails, and we shall demonstrate this by giving several counterexamples.
Our most significant contribution to the problem consists of a classification of the $n$-dimensional non-reflexive spaces
in which all the non-zero operators have rank greater than or equal to $2n-2$.
The result can be summed up as follows:

\begin{theo}\label{classtheointro}
Let $\calT$ be an $n$-dimensional reduced non-reflexive operator space between finite-dimensional vector spaces.
Assume that $\# \K>n \geq 3$ and that all the non-zero operators in $\calT$ have rank at least $2n-2$. \\
Then, $\calT$ is equivalent to a hyperplane of the twisted operator space associated with an LDB division algebra.
\end{theo}

This result is Corollary \ref{cornonreflexive} of Section \ref{nonreflexivesectionIII}. LDB division algebras are defined in Section \ref{LDBdef},
and their associated twisted operator spaces are defined in Section \ref{twisteddef}.

In short, LDB division algebras are non-associative division algebras in which the left-division mapping
is bilinear up to multiplication by a scalar-valued function. As it seems that no theory of those division algebras
has ever appeared in the literature, a large part of the article is devoted to their systematic study.
For the ease of read, we have divided our considerations on the minimal rank problem in a hyperplane of an LLD space
into three sections: in Section \ref{nonreflexivesectionI}, we quickly reprove the known results on the topic and
give elementary counterexamples to the Meshulam-\v Semrl conjecture; in Section \ref{nonreflexivesectionII},
we introduce the theory of left-division-bilinearizable division algebras
and show that it yields examples in which the upper bound $2n-2$ of Meshulam and \v Semrl is attained; in
Section \ref{nonreflexivesectionIII}, it is proved that our construction yields all the possible situations
in which the upper bound $2n-2$ of Meshulam and \v Semrl is attained for reduced spaces, up to equivalence and with
the usual cardinality assumptions on the underlying field.

\section{Preliminary results from matrix theory}\label{basiclemmassection}

\subsection{Field extension lemmas}

\begin{Def}
Let $\calS$ be a linear subspace of $\Mat_{m,n}(\K)$.
Given a basis $(A_1,\dots,A_s)$ of $\calS$ and indeterminates $\mathbf{x}_1,\dots,\mathbf{x}_s$,
the matrix $\mathbf{x}_1 A_1+\cdots+\mathbf{x}_s A_s$ of $\Mat_{m,n}\bigl(\K(\mathbf{x}_1,\dots,\mathbf{x}_s)\bigr)$
is called a \textbf{generic matrix} of $\calS$.

If we only assume that $A_1,\dots,A_s$ span $\calS$, then $\mathbf{x}_1 A_1+\cdots+\mathbf{x}_s A_s$
is called a \textbf{semi-generic matrix} of $\calS$.
\end{Def}

\noindent Note that the entries of $\mathbf{x}_1 A_1+\cdots+\mathbf{x}_s A_s$ are $1$-homogeneous polynomials in $\K[\mathbf{x}_1,\dots,\mathbf{x}_s]$.

\begin{lemme}[Generic rank lemma]\label{genericrank}
Let $\calS$ be a linear subspace of $\Mat_{m,n}(\K)$ with $\# \K>\urk(\calS)$, and $\mathbf{A}$ be a semi-generic matrix of $\calS$.
Then, $\urk(\calS)=\rk \mathbf{A}$.
\end{lemme}

\begin{proof}
Set $r:=\urk(\calS)$.
As specializing the indeterminates which constitute $\mathbf{A}$ yields all the possible matrices of $\calS$,
it is obvious, by looking at all the sub-determinants of $\mathbf{A}$, that $r \leq \rk \mathbf{A}$. \\
Let $I$ and $J$ be respective subsets of $\lcro 1,m\rcro$ and $\lcro 1,n\rcro$ with $\# I=\# J=r+1$.
For an $m\times n$ matrix $M$, we denote by $M(I\mid J)$ the $(r+1) \times (r+1)$ submatrix of $M$
obtained by selecting the rows indices in $I$ and the column indices in $J$.
Then, $\det\bigl(\mathbf{A}(I \mid J)\bigr)$ is an $(r+1)$-homogeneous polynomial with coefficients in $\K$.
By specializing the indeterminates $\mathbf{x}_1,\dots,\mathbf{x}_s$, we find the polynomial function
$M \in \calS \mapsto \det\bigl(M(I \mid J)\bigr)$, which vanishes everywhere on $\calS$
as $r+1>\urk(\calS)$. Since $r+1 \leq \#\K$, we deduce that $\det\bigl(\mathbf{A}(I \mid J)\bigr)=0$.
Varying $I$ and $J$ yields $\rk \mathbf{A} \leq r$.
\end{proof}

\begin{Def}
Given a linear subspace $\calS$ of $\Mat_{m,n}(\K)$ and a field extension $\L$ of $\K$, we define $\calS_\L$ as the $\L$-linear subspace of
$\Mat_{m,n}(\L)$ spanned by $\calS$.
\end{Def}

Any basis of $\calS$ as a $\K$-vector space is a basis of $\calS_\L$ as an $\L$-vector space.
In particular, the spaces $\calS$ and $\calS_\L$ have a common generic matrix.
Therefore, Lemma \ref{genericrank} yields:

\begin{lemme}[Field extension lemma]\label{fieldextension}
Let $\calS$ be a linear subspace of $\Mat_{m,n}(\K)$, with
$\# \K>\urk(\calS)$. Then, $\urk(\calS)=\urk(\calS_\L)$
for any field extension $\L$ of $\K$.
\end{lemme}

\subsection{The Flanders-Atkinson lemma}

\begin{lemme}[Flanders-Atkinson lemma]\label{Flanders}
Let $(m,n)\in (\N \setminus \{0\})^2$ and let $r \in \lcro 1,\min(m,n)\rcro$ be such that $\# \K > r$. \\
Set $J_r:=\begin{bmatrix}
I_r & [0]_{r \times (n-r)} \\
[0]_{(m-r) \times r} & [0]_{(m-r) \times (n-r)}
\end{bmatrix}$
and consider an arbitrary matrix
$M=\begin{bmatrix}
A & C \\
B & D
\end{bmatrix}$ of $\Mat_{m,n}(\K)$ with the same decomposition pattern. \\
Assume that $\urk(\Vect(J_r,M)) \leq r$.
Then,
$$D=0 \quad \text{and} \quad \forall k \in \N, \; BA^kC=0.$$
\end{lemme}

In the original result of Flanders \cite{Flanders}, only the conclusions $D=0$ and $BC=0$ were stated.
The more complete result is due to Atkinson \cite{AtkinsonPrim} and was later rediscovered by Fillmore, Laurie and Radjavi
\cite{Fillmoreetal} with stronger assumptions on the cardinality of the underlying field. Here is a new, simplified proof of this lemma.

\begin{proof}
Denote by $\L=\K((t))$ the quotient field of the power series ring over $\K$ with one indeterminate $t$.
Using the field extension lemma, we find that $\rk(J_r-tM) \leq r$.
Note that
$$J_r-tM=\begin{bmatrix}
I_r-tA & -tC \\
-tB & -tD
\end{bmatrix}$$
and that $I_r-tA$ is invertible, with $(I_r-tA)^{-1}=\underset{k=0}{\overset{+\infty}{\sum}}t^kA^k$. \\
Using Gaussian elimination on the block rows, we find that $J_r-tM$ is equivalent to
$$\begin{bmatrix}
I_r-tA & -tC \\
[0]_{(m-r) \times r} & -tD-(tB)(I_r-tA)^{-1}(tC)
\end{bmatrix}.$$
Since $\rk(J_r-tM)\leq r$ and $\rk(I_r-tA)=r$, it follows that
$$-tD-(tB)(I_r-tA)^{-1}(tC)=0,$$
whence
$$tD+\underset{k=0}{\overset{+\infty}{\sum}}t^{k+2}BA^kC=0.$$
This readily yields the claimed results.
\end{proof}

\subsection{The decomposition lemma}

\begin{lemme}\label{decompositionlemma}
Let $\calS$ be a linear subspace of $\Mat_{m,n}(\K)$ and assume that there exists a pair
$(r,s)$, with $1 \leq r\leq m$ and $1 \leq s\leq n$, such that
every matrix $M$ of $\calS$ splits up as
$$M=\begin{bmatrix}
[?]_{r \times s} & C(M) \\
B(M) & [0]_{(m-r) \times (n-s)}
\end{bmatrix}.$$
Consider a semi-generic matrix of $\calS$ of the form
$$\mathbf{M}=\begin{bmatrix}
[?]_{r \times s} & \mathbf{C} \\
\mathbf{B} & [0]_{(m-r) \times (n-s)}
\end{bmatrix}.$$
Assume that $\# \K > \urk \calS$.
Then, $\rk \mathbf{B}=\urk B(\calS)$, $\rk \mathbf{C}=\urk C(\calS)$ and
$$\urk(\calS) \geq \urk(B(\calS))+\urk(C(\calS)).$$
\end{lemme}

\begin{proof}
It is obvious that $\mathbf{B}$ and $\mathbf{C}$ are semi-generic matrices, respectively, of $B(\calS)$ and $C(\calS)$;
this yields the first result.
The last result comes from writing
$$\urk \calS=\rk \mathbf{M} \geq \rk \mathbf{B}+\rk \mathbf{C}=\urk B(\calS)+\urk C(\calS).$$
\end{proof}

\section{The minimal rank problem in an LLD operator space}\label{smallranksection}

\subsection{Basic theorems}

\begin{Def}
Let $\calS$ be a finite-dimensional subspace of $\calL(U,V)$.
An operator $\varphi \in \widehat{\calS}$ is called \textbf{rank-optimal} when it has
the greatest rank among the elements of $\widehat{\calS}$.
\end{Def}

Here is the basic method for finding a non-zero operator of small rank in an LLD space:
we take a rank-optimal operator $\varphi \in \widehat{\calS}$ and then we choose a non-zero operator $f \in \Ker \varphi$.
As we are about to see, $f$ has always finite rank. For large fields, the following lemma, which
is an easy consequence of the Flanders-Atkinson lemma, yields a powerful (and sharp) result:

\begin{lemme}\label{basiclemma}
Let $\calS$ be a finite-dimensional subspace of $\calL(U,V)$, and $\varphi$ be a rank-optimal element in
$\widehat{\calS}$. Assume that $\# \K > \rk \varphi$. Then, $\im f \subset \im \varphi$ for all $f \in \Ker \varphi$.
\end{lemme}

\noindent As a straightforward consequence, we have:

\begin{cor}[Meshulam and \v Semrl]\label{basiccor}
Let $\calS$ be an $n$-dimensional $c$-LLD subspace of $\calL(U,V)$. Assume that $\# \K>n-c$.
Then, there are respective subspaces $\calT$ and $V_0$ of $\calS$ and $U$
such that $\dim \calT \geq c$, $\dim V_0 \leq n-c$ and $\im f \subset V_0$ for every $f \in \calT$.
\end{cor}

\noindent The case $c=1$ in this corollary is the ``basic theorem" on LLD operator spaces:

\begin{theo}[Aupetit, Bre\v sar and \v Semrl]\label{basictheo}
Let $\calS$ be an $n$-dimensional LLD operator space.
Assume that $\# \K \geq n$.
Then, $\calS$ contains a non-zero operator $f$ such that $\rk f<n$.
\end{theo}

\begin{proof}[Proof of Lemma \ref{basiclemma}]
We use the same basic ideas as Meshulam and \v Semrl, the only difference being that
we rely explicitly on the Flanders-Atkinson lemma. \\
Assume first that $V$ is finite-dimensional, and set $r:=\rk \varphi$ and $m:=\dim V$.
Bases $\bfB$ and $\bfC$ of $\calS$ and $V$ may then be chosen so that
$$\Mat_{\bfB,\bfC}(\varphi)=J_r:=\begin{bmatrix}
I_r & [0]_{r \times (n-r)} \\
[0]_{(m-r) \times r} & [0]_{(m-r) \times (n-r)}
\end{bmatrix}.$$
Denote by $\calM$ the linear subspace of $\Mat_{m,n}(\K)$ obtained by representing the elements of $\widehat{\calS}$
in the bases $\bfB$ and $\bfC$. Then, $\urk(\calM)=r$ and hence the Flanders-Atkinson lemma shows that
every matrix of $\calM$ has the form $$\begin{bmatrix}
[?]_{r \times r} & [?]_{r \times (n-r)} \\
[?]_{(m-r) \times r} & [0]_{(m-r) \times (n-r)}
\end{bmatrix}.$$
The first $r$ vectors of $\bfC$ span $\im \varphi$ and the last $n-r$ ones of $\bfB$ span $\Ker \varphi$.
Therefore $\im f \subset \im \varphi$ for every $f \in \Ker \varphi$.

Let us now consider the general case. Assume that some $f \in \Ker \varphi$ does not satisfy
$\im f \subset \im \varphi$. Let us choose $y \in \im f \setminus \im \varphi$, and set
$V_0:=\im \varphi \oplus \K y$. Choose an idempotent $\pi \in \calL(V)$ with image $V_0$.
As $\im \varphi \subset V_0$, we have $\rk(\pi \circ \varphi)=\rk \varphi$ and hence $\pi \circ \varphi$ is rank-optimal in
$\pi \widehat{\calS}$. Representing the operators $g \in \calS \mapsto \pi(g(x)) \in V_0$, for $x \in U$,
by matrices in the same way as above, we find that $\pi(f(x)) \in \im (\pi \circ \varphi)=\im \varphi$ for all
$x \in U$, contradicting the fact that $y \not\in \im \varphi$.
\end{proof}

Note that the same method yields a useful result on defective spaces of matrices:

\begin{prop}\label{rangminmatriciel}
Let $\calM$ be a linear subspace of $\Mat_{m,n}(\K)$ with $\urk \calM<n$.
Assume that $\# \K> \urk \calM$.
Then, there exists a non-zero vector $x \in \K^n$ such that
$\dim \calM x \leq \urk \calM$.
\end{prop}

For small finite fields, Lemma \ref{basiclemma} fails: assume indeed that $\K$ has cardinality $q \leq r$,
let $n>r$, and consider the linear subspace of $\Mat_n(\K)$ spanned by the matrices
$$A=\begin{bmatrix}
I_r & [0] \\
[0] & [0]
\end{bmatrix} \quad \text{and} \quad B=\begin{bmatrix}
D & [0] \\
[0] & D'
\end{bmatrix},$$
where $D$ is an $r \times r$ diagonal matrix in which all the elements of $\K$ show up at least once on the diagonal,
and $D'$ is the $(n-r) \times (n-r)$ diagonal matrix $\Diag(1,0,\dots,0)$.
One checks that every linear combination of $A$ and $B$ has rank less than $r+1$.
However, $A$ is rank-optimal in $\Vect(A,B)$ and the $(r+1)$-th vector $e_{r+1}$ of the canonical basis of $\K^n$
belongs to $\Ker A$ although $Be_{r+1} \not\in \im A$.

\vskip 3mm
For finite fields, the known results are weaker, but nevertheless interesting.
Meshulam and \v Semrl \cite{MeshulamSemrlPJM} proved that
every $n$-dimensional $c$-LLD operator space contains a non-zero element $f$ with $\rk f \leq n-c$.
We shall use their counting method to improve the result a little bit:

\begin{prop}\label{finitefieldcase}
Let $\calS$ be an $n$-dimensional linear subspace of $\calL(U,V)$, and let $c \in \lcro 1,n-1\rcro$.
Assume that $\calS$ is $c$-LLD and that $\K$ is a finite field with $q$ elements.
Then, at least $q^c$ elements of $\calS$ have rank less than or equal to $n-c$.
\end{prop}

\begin{proof}
As in \cite{MeshulamSemrlPJM}, we assume first that $U$ is finite-dimensional. Set $p:=\dim U$.
Denote by $N$ the cardinality of the set of all elements of $\calS$ which have rank less than or equal to $n-c$.
Consider the set
$$A:=\bigl\{(f,x)\in \calS \times U : \; f(x)=0\bigr\}$$
and its two canonical projections $\pi_1 : (f,x) \mapsto f$ and $\pi_2 : (f,x) \mapsto x$.
Then, $A$ is finite and we may estimate its cardinality in two different ways:
\begin{itemize}
\item For every non-zero vector $x \in U$, the assumptions show that $\dim \pi_2^{-1}\{x\} \geq c$ and therefore $\# \pi_2^{-1}\{x\} \geq q^c$.
Taking the zero vector into account yields:
$$\# A \geq q^n+q^c(q^p-1).$$
\item For every $f \in \calS$, the rank theorem shows that $\# \pi_1^{-1}\{f\}=q^{p-\rk f}$; thus:
either $f=0$ and then $\# \pi_1^{-1}\{f\}=q^p$;
or $1 \leq \rk f \leq n-c$ and then $\# \pi_1^{-1}\{f\} \leq q^{p-1}$;
or $\rk f>n-c$ and then $\# \pi_1^{-1}\{f\} \leq q^{p+c-n-1}$.
This yields:
$$\# A \leq q^p+(N-1)q^{p-1}+(q^n-N)q^{p+c-n-1}.$$
\end{itemize}
We deduce that
$$q^n+q^c(q^p-1) \leq q^p+(N-1)q^{p-1}+(q^n-N)q^{p+c-n-1}$$
and hence
$$(q^{p-1}-q^{p+c-n-1})(N+1-q^c) \geq q^n-q^c+q^{p+2c-n-1}-q^{p+c-n-1}+q^{c+p}-2q^{p+c-1}-q^p+2q^{p-1}.$$
However $q^n>q^c$, $q^{p+2c-n-1} \geq q^{p+c-n-1}$ and
$$q^{c+p}-2q^{p+c-1}-q^p+2q^{p-1}=(q-2)(q^{p+c-1}-q^{p-1}) \geq 0.$$
As $q^{p-1}>q^{p+c-n-1}$, it follows that $N>q^c-1$. This finishes the proof in the case $U$ is finite-dimensional.

Now, let us move on to the general case. Assume that more than $q^n-q^c$ operators $f_1,\dots,f_s$ in $\calS$
have rank greater than $n-c$ (possibly with infinite rank). For each $i \in \lcro 1,s\rcro$, choose an $(n-c+1)$-dimensional
linear subspace $U_i$ which intersects
$\Ker f_i$ trivially. For every non-zero operator $f$ of $\calS$, choose some $x_f \in U$ such that $f(x) \neq 0$.
Set $U_0:=\Vect\bigl\{x_f \mid f \in \calS \setminus \{0\}\bigr\}+\underset{i=1}{\overset{s}{\sum}} U_i$,
and note that $U_0$ is finite-dimensional and that $\Psi : f \in \calS \mapsto f_{|U_0} \in \calL(U_0,V)$ is a one-to-one linear map.
Noting that $\Psi(\calS)$ is a $c$-LLD subspace of $\calL(U_0,V)$
and that $\Psi(f_1),\dots,\Psi(f_s)$ are all distinct with rank greater than $n-c$,
one finds a contradiction with the first part of the proof.
\end{proof}

Note that this does not prove that enough small rank operators may be found in a common $(n-c)$-dimensional subspace of $\calS$
as in Corollary \ref{basiccor}.
On the other hand, for an arbitrary field, we can give an upper bound on the rank that is not as tight
but holds for all operators in a specific $(n-c)$-dimensional subspace:

\begin{prop}\label{AmitsurlikeProp}
Let $\calS$ be an $n$-dimensional LLD subspace of $\calL(U,V)$, and $\varphi$ be a rank-optimal element of $\widehat{\calS}$, with rank $r$. \\
Then, $\rk f \leq \dbinom{r+1}{2}+r(n-r)$ for all $f \in \Ker \varphi$.
\end{prop}

As $r \mapsto \dbinom{r+1}{2}+r(n-r)$ is increasing on $[1,n]$ (its derivative being $r \mapsto -r+n+\frac{1}{2}$),
we deduce the following corollary:

\begin{cor}\label{AmitsurlikeCor}
Let $\calS$ be an $n$-dimensional $c$-LLD subspace of $\calL(U,V)$. \\
Then, there are subspaces $\calT$ and $V_0$, respectively, of $\calS$ and $V$ such that
$\dim \calT \geq n-c$, $\dim V_0 \leq \dbinom{n-c+1}{2}+c(n-c)$ and $\im f \subset V_0$ for all $f \in \calT$.
\end{cor}

Note that the upper bound $\dbinom{n-c+1}{2}+c(n-c)$ is of the same order of
magnitude as the upper bound in Amitsur's lemma
\cite[Lemma 1]{Amitsur}.

\begin{proof}[Proof of Proposition \ref{AmitsurlikeProp}]
We start with the case when $V$ is finite-dimensional.
We may then find respective bases $\bfB$ and $\bfC$ of $\calS$ and $V$ such that
$$M_{\bfB,\bfC}(\varphi)=J_r:=\begin{bmatrix}
I_r & [0]_{r \times (n-r)} \\
[0]_{(m-r) \times r} & [0]_{(m-r) \times (n-r)}
\end{bmatrix}.$$
Denote by $\calM$ the vector space of the matrices associated with $\widehat{\calS}$ in the bases $\bfB$ and $\bfC$.
Split every $M \in \calM$ as
$$M=\begin{bmatrix}
L(M)+U(M) & B(M) \\
[?]_{(m-r) \times r} & D(M)
\end{bmatrix},$$
where $L(M)$, $U(M)$, $B(M)$ and $D(M)$ are, respectively, $r \times r$, $r \times r$, $r \times (n-r)$ and $(m-r) \times (n-r)$ matrices,
with $L(M)$ lower-triangular and $U(M)$ strictly upper-triangular; write also $E(M):=\begin{bmatrix}
B(M) \\
D(M)
\end{bmatrix}$.
For every $M \in \calM$ such that $L(M)=0$ and $B(M)=0$, combining inequality $\rk (J_r+M) \leq r$ with the fact that
$I_r+U(M)$ is invertible yields $D(M)=0$ and therefore $E(M)=0$.
Using the rank theorem for the map $M \mapsto E(M)$, we deduce that
$$\dim E(\calM) \leq \dim L(\calM)+\dim B(\calM) \leq \binom{r+1}{2}+r(n-r),$$
and our claim follows as $\dim \calM X=\dim E(\calM)Y \leq \dim E(\calM)$ for all $X =\begin{bmatrix}
[0]_{r \times 1} \\
Y
\end{bmatrix} \in \{0\} \times \K^{n-r}$.

For the general case, assume that some $f \in \Ker \varphi$ satisfies $\rk f>\dbinom{r+1}{2}+r(n-r)$,
choose a finite-dimensional linear subspace $V_1$ of $\im f$ such that $\dim V_1 >\dbinom{r+1}{2}+r(n-r)$,
and set $V_0:=\im \varphi+V_1$. Choose an idempotent $\pi$ of $\calL(V)$ with range $V_0$.
Then, with the same line of reasoning as in our proof of Lemma \ref{basiclemma}, one deduces from the
finite-dimensional case
that $\rk (\pi \circ f) \leq \dbinom{r+1}{2}+r(n-r)$, which contradicts the fact that $V_1 \subset \im (\pi \circ f)$.
\end{proof}

With the above results, we rediscover the following known theorem:

\begin{theo}[Bre\v sar,  Meshulam and \v Semrl]\label{basictheogeneral}
Every $n$-dimensional LLD operator space contains a non-zero operator with rank less that $n$.
\end{theo}

The following application of Theorem \ref{basictheogeneral} was kindly suggested to us by Jean-Pierre Barani:

\begin{prop}\label{Baranitheorem}
Let $\L$ be a field extension of finite degree $d$ over $\K$, and $V$ be a finite-dimensional $\L$-vector space.
Then, for every $\K$-linear subspace $W$ of $V$ whose dimension is a multiple of $d$,
there is an $\L$-linear subspace $W'$ of $V$ such that $V=W\oplus W'$.
\end{prop}

\begin{proof}
The result is obvious if $W=V$. Assume now that $W \subsetneq V$, so that $\dim_\K (V/W) \geq d$.
In that case, we show that there is a non-zero vector $x \in V$ such that $\L x \cap W=\{0\}$.
To prove the claimed result, we consider the space $\calS$ of all $\K$-linear operators
$$\widehat{\lambda} :\quad  x \in V \longmapsto [\lambda x] \in V/W, \quad \text{with $\lambda \in \L$.}$$
 Obviously, for all $\lambda \in \L \setminus \{0\}$, the operator $\widehat{\lambda}$ is surjective,
 whence its rank equals $\dim_\K(V/W)\geq d$. In particular, all these operators are non-zero, leading to $\dim_\K \calS=d$.
 We deduce from Theorem \ref{basictheogeneral} that $\calS$ is not LLD, which yields
 a vector $x \in V$ such that $\lambda x \not\in W$ for all $\lambda \in \L \setminus \{0\}$.
 Thus, $\L x \cap W=\{0\}$, as claimed.
 Noting that $\dim_\K(\L x\oplus W)=d+\dim_\K W$, the result is then easily obtained by downward induction on $\dim_\K W$.
 \end{proof}

An important consequence of Theorem \ref{basictheogeneral} is the known fact that any LLD operator space contains a non-zero operator
with finite rank. In \cite{Larson}, Larson uses this fact to prove that in an LLD operator space, the linear
subspace of finite rank operators is LLD itself, to the effect that every minimal LLD operator space contains only finite rank operators.
We generalize his results to $c$-LLD spaces:

\begin{prop}\label{finiterankprop}
Let $\calS$ be a $c$-LLD subspace of $\calL(U,V)$, and denote by $\calS_F$
its linear subspace of finite rank operators. Then, $\calS_F$ is $c$-LLD.
\end{prop}

\begin{proof}
Our method is essentially similar to Larson's.
We extend a basis $\calB=(f_1,\dots,f_p)$ of $\calS_F$ into a basis $(f_1,\dots,f_n)$ of $\calS$.
Let $x \in U$.
Set $V_1:=\Vect(f_1(x),\dots,f_n(x))$, choose an idempotent
$\pi$ in $\calL(V)$ with kernel $V_1$, and set $U_0:=\underset{k=1}{\overset{p}{\bigcap}} \Ker f_k$.
Note that $U_0$ has finite codimension in $U$ since $f_1,\dots,f_p$ have finite rank.
Consider the space of operators $\calT:=\bigl\{\pi \circ f_{|U_0} \mid f \in \Vect(f_{p+1},\dots,f_n)\bigr\}$.
For any non-zero operator $f \in \Vect(f_{p+1},\dots,f_n)$, the operator $\pi \circ f_{|U_0}$ does not have finite rank as $f \not\in \calS_F$;
therefore $\dim \calT=n-p$ and no non-zero operator in $\calT$ has finite rank. \\
We deduce that $\calT$ is not LLD, yielding some $y \in U_0$
such that $\dim \Vect\bigl(\pi\bigl(f_{p+1}(y)\bigr),\dots,\pi\bigl(f_n(y)\bigr)\bigr)=n-p$.
Now, set $V_0:=V_1+\Vect(f_{p+1}(y),\dots,f_n(y))$, $m:=\dim V_0$ and $q:=\dim V_1$.
Note that given a basis $(v_1,\dots,v_q)$ of $V_1$, the family $\mathcal{C}=(v_1,\dots,v_q,f_{p+1}(y),\dots,f_n(y))$ is a basis of
$V_0$. Choose finally an idempotent $\pi'$
of $\calL(V)$ with range $V_0$.
Then, for any $z \in U$, the matrix $M(z)$ of $f \mapsto \pi'(f(z))$ in the bases $\bfB$ and $\bfC$ has rank at most $n-c$,
whereas
$$M(x)=\begin{bmatrix}
A & [?]_{q \times (n-p)} \\
[0]_{(n-p) \times p} & [0]_{(n-p) \times (n-p)}
\end{bmatrix} \quad \text{and} \quad
M(y)=\begin{bmatrix}
[0]_{q \times p} & [0]_{q \times (n-p)} \\
[0]_{(n-p) \times p} & I_{n-p}
\end{bmatrix}$$
where $A$ is $q \times p$ matrix with $\rk A=\rk (f_1(x),\dots,f_p(x))$.
Then,
$$M(x+y)=\begin{bmatrix}
A & [?]_{q \times (n-p)} \\
[0]_{(n-p) \times p} & I_{n-p}
\end{bmatrix}.$$
As $\rk M(x+y) \leq n-c$, we deduce that
$$\rk (f_1(x),\dots,f_p(x))=\rk A \leq \rk M(x+y)-\rk I_{n-p} \leq p-c.$$
Therefore, $\calS_F$ is $c$-LLD.
\end{proof}

\begin{prop}\label{minimalisfinitedimensional}
Let $\calS$ be a finite-dimensional reduced minimal $c$-LLD subspace of $\calL(U,V)$.
Then, $U$ and $V$ are finite-dimensional.
\end{prop}

\begin{proof}
As $\calS$ is minimal among $c$-LLD subspaces of $\calL(U,V)$,
Proposition \ref{finiterankprop} yields $\calS_F=\calS$, meaning that every operator in $\calS$ has finite rank. As $\calS$
is finite-dimensional, one finds a basis $(f_1,\dots,f_n)$ of $\calS$, and then
$\underset{k=1}{\overset{n}{\bigcap}} \Ker f_k$ has finite codimension in $U$, while
$\underset{k=1}{\overset{n}{\sum}} \im f_k$ has finite dimension. As $\calS$ is reduced, this shows that
$U$ and $V$ are finite-dimensional.
\end{proof}

\subsection{On a refinement of Meshulam and \v Semrl}

In \cite{MeshulamSemrlPJM}, Meshulam and \v Semrl proved the following beautiful refinement of Theorem \ref{basictheo}:

\begin{theo}[Meshulam and \v Semrl]\label{refinedMStheo}
Let $\calS$ be an $n$-dimensional LLD subspace of $\calL(U,V)$, with $\# \K \geq n+2$.
Then, either $\calS$ contains a non-zero element of rank less than $n-1$, or else
all the non-zero elements of $\calS$ have rank $n-1$.
\end{theo}

This result has been widely used in subsequent works of Chebotar, Meshulam and \v Semrl \cite{ChebotarSemrl,
MeshulamSemrlPJM,MeshulamSemrlLAA}.
In \cite{MeshulamSemrlPJM}, it is proven first for infinite fields with an argument from algebraic geometry,
and then for finite fields with a counting argument. We shall improve the result as follows, with a unified proof:

\begin{theo}\label{refinedMStheogeneralized}
Let $\calS$ be an $n$-dimensional LLD subspace of $\calL(U,V)$.
Assume that $\# \K \geq n$ and that $\calS$ is not $2$-LLD.
Then:
\begin{enumerate}[(i)]
\item Either $\rk f \leq n-1$ for every $f \in \calS$.
\item Or, for every rank-optimal operator $\varphi \in \widehat{\calS}$
and every $f \in \Ker \varphi$, one has $\rk f<n-1$.
\end{enumerate}
\end{theo}

As Corollary \ref{basiccor} shows that every $2$-LLD space of operators contains a non-zero operator of rank less than $n-1$,
Theorem \ref{refinedMStheogeneralized} yields:

\begin{theo}\label{refinedMStheobetter}
Let $\calS$ be an $n$-dimensional LLD subspace of $\calL(U,V)$, with $\# \K \geq n$.
Then, either $\calS$ contains a non-zero element of rank less than $n-1$, or else
all the non-zero elements of $\calS$ have rank $n-1$.
\end{theo}

\begin{proof}[Proof of Theorem \ref{refinedMStheogeneralized}]
Set $r:=n-1$ for convenience.
Assume first that $V$ is finite-dimensional.
Let $\varphi$ be a rank-optimal element of $\widehat{\calS}$, so that $\rk \varphi=r$, and
assume that $\Ker \varphi$ contains a non-zero element $f$ with $\rk f \geq r$.
Then $\im f=\im \varphi$ as we already know from Lemma \ref{basiclemma} that $\im f \subset \im \varphi$.
We may then find respective bases of $\calS$ and $V$ such that
the matrix of $\varphi$ in those bases is
$$J:=\begin{bmatrix}
I_r & [0]_{r \times 1} \\
[0]_{(m-r) \times r} & [0]_{(m-r) \times 1}
\end{bmatrix}, \quad \text{where $m:=\dim V$.}$$
Denote by $\calM$ the vector space formed by the matrices associated with the operators of $\widehat{\calS}$ in the above bases.
By Lemma \ref{Flanders} and our assumptions on $\calS$, we see that $\calM$ has the following properties:
\begin{enumerate}[(a)]
\item $\urk(\calM)=r$.
\item For every $C \in \K^r$, the subspace $\calM$ contains a matrix of the form
$$\begin{bmatrix}
[?]_{r \times r} & C \\
[?]_{(m-r) \times r} & [0]_{(m-r) \times 1}
\end{bmatrix}.$$
\end{enumerate}
Let us say that a vector $X \in \K^n$ is \textbf{$\calM$-good} when it belongs to the kernel
of a rank $r$ matrix of $\calM$. By Lemma \ref{basiclemma}, one has $\dim \calM X \leq r$ for every such $X$.
For infinite fields, our basic idea is to show that the set of all $\calM$-good vectors is algebraically dense in $\K^n$.

Consider a basis $\bfD$ of $\calM$, denote by $\bfC$ the canonical basis of $\K^m$,
and, for $X \in \K^n$, consider the linear operator
$\check{X} : M \in \calM \mapsto MX \in \K^m$. \\
Set $p:=\dim \calM$.
Fix arbitrary finite subsets $I$ of $\lcro 1,m\rcro$ and $J$ of $\lcro 1,p\rcro$, both with cardinality $n$, and, for
a matrix $N \in \Mat_{m,p}(\K)$, denote by $N(I\mid J)$ the matrix of $\Mat_n(\K)$ obtained from $N$
by selecting the rows indexed over $I$ and the columns indexed over $J$.
Then $X \in \K^n \mapsto \det \Mat_{\bfD,\bfC}(\check{X})(I\mid J)$ is
a polynomial function associated with an $n$-homogeneous polynomial $G \in \K[\mathbf{x_1},\dots,\mathbf{x_n}]$, and
it vanishes at every $\calM$-good vector. We wish to show that $G$ vanishes everywhere on $\K^n$.

Obviously, $G$ is unchanged by replacing the ground field $\K$ with one of its extensions $\L$
(as $\bfD$ is also a basis of the $\L$-vector space $\calM_\L$).
Moreover, the space $\calM_\L$ still has upper-rank $r$ (see Lemma \ref{fieldextension}), and hence $G$ vanishes at
every $\calM_\L$-good vector.

Let $X \in \K^r$. Using point (b) above, we may find a matrix in $\calM$ of the form
$$M=\begin{bmatrix}
A & X \\
[?] & [?]
\end{bmatrix}, \quad \text{where $A \in \Mat_r(\K)$.}$$
We now work with $\L=\K((t))$, the quotient field of the power series ring over $\K$ with one indeterminate $t$.
The space $\calM_\L$ contains
$$J-tM=\begin{bmatrix}
I_r-tA & -tX \\
[?] & [?]
\end{bmatrix}.$$
However $I_r-tA$ is invertible, with inverse $\underset{k=0}{\overset{+\infty}{\sum}}t^k A^k$.
As $\rk(J-tM) \leq r$, we deduce that
$$\Ker (J-tM)=\Ker \begin{bmatrix}
I_r-tA & -tX
\end{bmatrix}.$$
Therefore, the vector $\begin{bmatrix}
t(I_r-tA)^{-1}X \\
1
\end{bmatrix}$ belongs to the kernel of $J-tM$, which shows that it is $\calM_\L$-good.
Thus,
\begin{equation}\label{powerseries}
G\bigl(t(I_r-tA)^{-1}X\, ,\,1\bigr)=0.
\end{equation}
Let us split
$$G=\underset{k=0}{\overset{n}{\sum}}G_k(\mathbf{x_1},\dots,\mathbf{x_{n-1}})\,(\mathbf{x_n})^{n-k},$$
where $G_k$ is a $k$-homogeneous polynomial for all $k \in \lcro 0,n\rcro$.
As $(0,\dots,0,1)$ is $\calM$-good, we already have $G_0=0$. \\
Assume now that $G \neq 0$, and set $d:=\min \bigl\{k \in \lcro 1,n\rcro : \; G_k\neq 0\bigr\}$.
Expanding identity \eqref{powerseries} then yields
$$t^d \,G_d(X)=0 \quad \text{mod. $t^{d+1}$,}$$
and hence $G_d(X)=0$. Varying $X$ then yields $G_d=0$ as $\# \K \geq n \geq d$.
This contradiction shows that $G=0$.

Varying $I$ and $J$, we deduce that $\rk \check{X} \leq r$ for all $X \in \K^n$,
and hence $\rk f < n$ for all $f \in \calS$.

We finish by reducing the general case to the finite-dimensional one.
Assume that $\calS$ contains an operator $g$ such that $\rk g>n-1$.
Let $\varphi$ be a rank-optimal operator in $\widehat{\calS}$.
Then, we may choose a finite-dimensional linear subspace $V_1$ of $V$ such that $V_1 \subset \im g$ and $\dim V_1>n-1$.
Therefore, $V_0:=\im \varphi+V_1$ is finite-dimensional, and we may choose an idempotent $\pi \in \calL(V)$
with range $V_0$. Using the above matrix method for the operator space
$\bigl\{\pi \circ \psi\mid \psi \in \widehat{\calS}\bigr\}$, we find that
$\rk (\pi \circ f)<n-1$ for all $f \in \Ker (\pi \circ \varphi)$.
Therefore, for every $f \in \Ker \varphi$, one concludes that $\rk (f)=\rk (\pi \circ f)<n-1$
as $\im f \subset \im \varphi$.
\end{proof}

For $c$-LLD operator spaces, we have the following easy corollary:

\begin{cor}\label{refinedcor}
Let $\calS$ be an $n$-dimensional $c$-LLD subspace of $\calL(U,V)$.
Assume that $\# \K > n-c$ and that $\calS$ is not $(c+1)$-LLD.
Then:
\begin{enumerate}[(i)]
\item Either $\rk f \leq n-c$ for all $f \in \calS$.
\item Or, for every rank-optimal $\varphi \in \widehat{\calS}$
and every $f \in \Ker \varphi$, one has $\rk f<n-c$.
\end{enumerate}
\end{cor}

\begin{proof}
Assume that (i) is false, and choose some $h \in \calS$ with $\rk h>n-c$.
Let $\varphi$ be a rank-optimal operator in $\widehat{\calS}$. Then, $\rk \varphi=n-c$
and Lemma \ref{basiclemma} shows that $\varphi(h) \neq 0$. \\
Let $f \in \Ker \varphi \setminus \{0\}$.
We may choose a splitting $\calS=\calS'\oplus \Ker \varphi$ such that $h \in \calS'$,
and set $\calT:=\calS' \oplus \K f$. Then, $\calT$ is an $(n-c+1)$-dimensional LLD operator space
and $\varphi_{|\calT}$ is rank-optimal in $\widehat{\calT}$ as $\varphi(\calS')=\im \varphi$ is $(n-c)$-dimensional.
As $h \in \calT$ and $\rk h>n-c$, Theorem \ref{refinedMStheogeneralized} applied to $\calT$ yields $\rk f<n-c$.
Therefore, property (ii) holds.
\end{proof}

As the following example shows, it is hard to come up with a significant strengthening of
Corollary \ref{refinedcor}.

\begin{ex}
Let $p$ be an even integer with $p>2$. Set $n:=\dbinom{p}{2}$.
Consider the space $\calS=\Mata_p(\K)$, naturally seen as an $n$-dimensional linear subspace of $\calL(\K^p)$.
We endow $\K^p$ with its canonical non-degenerate symmetric bilinear form $(X,Y) \mapsto X^TY$, for which orthogonality is now considered.
For every non-zero $x \in \K^p$, the space $\calS x=\{x\}^\bot$ is $(p-1)$-dimensional.
It follows that $\calS$ is a $c$-LLD operator space for $c:=\dbinom{p}{2}-(p-1)=\dbinom{p-1}{2}$
and that every non-zero operator in $\widehat{\calS}$ is rank-optimal. \\
However:
\begin{itemize}
\item $\calS$ is obviously not $(c+1)$-LLD.
\item As $p$ is even, there is an invertible matrix in $\Mata_p(\K)$; therefore, some
operator in $\calS$ has rank greater than $n-c=p-1$.
\item For every $x \in \K^p \setminus \{0\}$, the kernel of $\varphi : f \mapsto f(x)$ consists of
all the alternating matrices $A \in \Mata_p(\K)$ for which $Ax=0$: one checks that
$\{Ay \mid A \in \Ker \varphi, \; y \in \K^p\}=\{x\}^\bot=\im \varphi$.
\end{itemize}
Therefore, even if conclusion (i) of Corollary \ref{refinedcor} does not hold for a given space $\calS$,
the set $\bigl\{f(y) \mid f \in \Ker \varphi, \;y \in U\bigr\}$ may equal $\im \varphi$ for some rank-optimal
operator $\varphi \in \widehat{\calS}$.
\end{ex}

\subsection{A note on Amitsur's lemma}\label{Amitsurnote}

Amitsur's work on rings with polynomial identities \cite{Amitsur} was one of the reasons why the minimal rank problem
was studied in the first place. In this work, Amitsur actually needed a generalization to
skew fields (Lemma 3 in \cite{Amitsur}), which boils down to the following statement:

\begin{lemme}[Amitsur's lemma]
Let $A$ be an abelian group, $V$ be a left vector space over a skew field $D$,
and $f_1,\dots,f_n$ be group homomorphisms from $A$ to $V$ such that
$(f_1(x),\dots,f_n(x))$ is linearly dependent over $D$ for all $x \in A$. Then, for some
$(\alpha_1,\dots,\alpha_n) \in D^n \setminus \{0\}$, the operator $\alpha_1 f_1+\cdots+\alpha_n f_n$ has finite rank less than or equal to
$\dbinom{n+1}{2}-2$.
\end{lemme}

Here, we wish to make a couple of remarks on that lemma. First of all, by following Amitsur's proof, one actually
comes up with the improved upper bound $\dbinom{n}{2}$. Indeed, in the notation of the proof of Lemma 3.1 of \cite{Amitsur},
inequality $\dim U_1 \leq \dim U_0+\dim \calT v_1 \leq \dim U_0+\tau$ can be replaced with
$\dim U_1 \leq \dim U_0+\dim \calT v_1-1\leq \dim U_0+\tau-1$ since it is assumed that the vectors $T_1 v_1,\dots,T_\tau v_1$ are linearly
dependent mod.\ $U_0$, and one can proceed by induction to obtain the claimed upper bound.

Secondly, Meshulam and \v Semrl claimed (see p. 452 of \cite{MeshulamSemrlPJM})
that Bre\v sar and \v Semrl's proof of Theorem \ref{basictheo} actually applies to infinite skew fields and, together with the
corresponding theorem for finite fields, yields the improved upper bound $n-1$ in Amitsur's lemma.
This is doubtful however as their method would require that, for every
$n \times n$ square matrix $M$ with entries in $D$, there should be an element of the prime subfield of $D$
which is not a left-eigenvalue of $M$, a statement which is obviously false for skew fields of small positive characteristic.

However, in order to prove Theorem 4.1 of \cite{MeshulamSemrlPJM}, it is not necessary to improve
the upper bound in the general context of Amitsur's lemma, but only in the one where
$D$ is finite-dimensional over its center $C$, the abelian group $A$ is endowed with a structure of $C$-vector space,
and the group homomorphisms $f_1,\dots,f_n$ are actually $C$-linear. In that situation, we can indeed prove:

\begin{lemme}
Let $D$ be a division algebra which is finite-dimensional over its center $C$.
Let $U$ be a $C$-vector space, $V$ be a $D$-left vector space, and $f_1,\dots,f_n$ be $C$-linear maps from $U$ to $V$.
Then, some non-trivial linear combination of $f_1,\dots,f_n$ over $D$ has rank at most $n-1$.
\end{lemme}

\begin{proof}
If $C$ is finite, then $D$ is also finite, and hence $D=C$ as every finite skew field is commutative.
In that case, the result follows from Proposition \ref{finitefieldcase}, or, alternatively, from Corollary 2.3 of \cite{MeshulamSemrlPJM}.

Let us now assume that $C$ is infinite. Then, Bre\v sar and \v Semrl's Lemma 2.1 of \cite{BresarSemrl}
may be adapted by taking $W$ as a vector space over $D$ and by stating that at least one $\alpha$ in $C$
is such that $w_1+\alpha z_1,\dots,w_r+\alpha z_r$ are linearly independent:
indeed, its proof ultimately rests on the fact that a matrix of $\Mat_r(D)$ has finitely many left-eigenvalues in $C$,
which holds true because the set of left-eigenvalues of such a matrix is a finite union of conjugacy classes
in $D$, and those associated with central scalars are singletons.
From there, the proof of Theorem 2.2 of \cite{BresarSemrl} may be adapted effortlessly, the only difference being that we require that $\alpha$
belongs to $C$.
\end{proof}

With a similar line of reasoning, one sees that the upper bound $n-1$ holds in the conclusion of Amitsur's lemma
provided that the skew field be of characteristic zero: in the above proof, we replace $C$ with the ring of integers,
naturally seen as a subring of $D$. For skew fields with positive characteristic, we do not know whether a better upper bound
than $\dbinom{n}{2}$ can be obtained in general.

\section{The minimal rank problem in a non-reflexive operator space (I)}\label{nonreflexivesectionI}

In this section, all the vector spaces are assumed to be finite-dimensional.

\subsection{Definitions and aims}

\begin{Def}
Let $\calS$ be a linear subspace of $\calL(U,V)$. The reflexive closure $\calR(\calS)$
of $\calS$ is the set of all linear operators $g \in \calL(U,V)$ for which $g(x) \in \calS x$ for every $x \in U$. \\
We say that $\calS$ is (\textbf{algebraically}) \textbf{reflexive} when $\calS=\calR(\calS)$.
\end{Def}

Note that the reflexive closure of $\calS$ is the smallest reflexive linear subspace of $\calL(U,V)$ which contains
$\calS$.

There is an obvious link between LLD spaces and non-reflexive spaces:
let $\calS$ be a non-reflexive operator space, and $f$ be an operator in $\calR(\calS) \setminus \calS$.
Then, $\calS \oplus \K f$ is obviously LLD and has dimension $\dim \calS+1$.
Moreover, $\calR(\calS)$ is $c$-LLD for $c:=\dim \calR(\calS)-\dim \calS$.

The converse statements are not true in the sense that, for a non-LLD operator space
$\calS$, there may be an operator $f$ outside of $\calR(\calS)$ for which $\calS \oplus \K f$ is LLD.
For example, one takes $\calS=\K \pi_1$ and $f=\pi_2$, where $\pi_1$ and $\pi_2$ are the two canonical
projections of $\K^2$ onto $\K$. Then, $\calS \oplus \K \pi_2$ is LLD, but surely
$\pi_2$ does not belong to $\calR(\calS)$, and one can even prove that $\calS$ is reflexive.

\vskip 3mm
If we take a non-reflexive space $\calS$ and $f \in \calR(\calS) \setminus \calS$
as above, then $\calS$ appears as a linear hyperplane of the LLD space $\calS \oplus \K f$.
Thus, if we have general results on the minimal rank of a hyperplane of an LLD space, then
those results yield theorems on the minimal rank of a non-reflexive space of operators,
which leads to sufficient conditions for reflexivity.

\subsection{Two results of Meshulam and \v Semrl}

Here, we shall review and reprove results of Meshulam and \v Semrl \cite{MeshulamSemrlLAA} on the minimal rank in a non-reflexive space
of operators.

\begin{theo}[Meshulam and \v Semrl \cite{MeshulamSemrlLAA}]\label{minranktheo1}
Let $\calS$ be an $(n+1)$-dimensional LLD subspace of $\calL(U,V)$, with $\# \K>n \geq 2$, and $\calT$ be a hyperplane of $\calS$.
Then, $\mrk \calT \leq 2n-2$.
\end{theo}

This yields the following result for algebraic reflexivity:

\begin{cor}\label{minrankcor1}
Let $\calS$ be a non-reflexive $n$-dimensional subspace of $\calL(U,V)$.
Assume that $\# \K>n \geq 2$.
Then, $\mrk \calS \leq 2n-2$.
\end{cor}

It follows of course that an $n$-dimensional subspace of $\calL(U,V)$
is reflexive whenever all its non-zero operators have rank greater than $2n-2$.
This was a recent success of the strategy of using LLD operator techniques to analyze reflexivity.

We recall the proof of Theorem \ref{minranktheo1}, as its line of reasoning will
be crucial in the study of the critical case.

\begin{proof}[Proof of Theorem \ref{minranktheo1}]
We choose a rank-optimal operator $\varphi \in \widehat{\calS}$, whose rank we denote by $r$.
We know that $\rk f \leq r$ for all $f \in \Ker \varphi$. If
$\Ker \varphi \cap \calT \neq \{0\}$, then we deduce that
$\mrk \calT \leq r \leq n \leq 2n-2$.
Assume now that $\Ker \varphi \cap \calT=\{0\}$; since $\dim \Ker \varphi>0$, this yields
$r=n$, $\calS=\calT \oplus \Ker \varphi$ and $\Ker \varphi=\K g$ for some operator $g \in \calS \setminus \{0\}$.
If $\rk g=n$, then we deduce from Theorem \ref{refinedMStheogeneralized} that $\rk f \leq n$ for all non-zero operators $f \in \calS$,
and in particular $\mrk \calT \leq n \leq 2n-2$.
Assume now that $\rk g \leq n-1$. Set $s:=\rk g$.
Let us choose a basis $\calB$ of $\calS$ that is adapted to the decomposition
$\calS=\calT \oplus \Ker \varphi$, together with a basis $\calC$ of $V$ in which the first vectors form a basis of $\im g$.
Denote by $\calM$ the space of matrices representing the operators of the dual operator space
$\widehat{\calS}$ in the bases $\calB$ and $\calC$. Then, setting $m:=\dim V$, every $M \in \calM$ splits up as
$$M=\begin{bmatrix}
[?]_{s \times n} & [?]_{s \times 1} \\
B(M) & [0]_{(m-s) \times 1}
\end{bmatrix},$$
where, by Lemma \ref{decompositionlemma}, $B(\calM)$ is a linear subspace of $\Mat_{m-s,n}(\K)$ with $\urk B(\calM) \leq n-1$.
Proposition \ref{rangminmatriciel} yields a non-zero vector $y \in \K^n$ such that
$\dim B(\calM)y \leq n-1$. Denoting by $f$ the (non-zero) vector of $\calT$ whose coordinate matrix in the chosen basis of $\calT$ is $y$,
we deduce that $\rk f \leq s+\dim B(\calM)y \leq 2n-2$, which completes the proof.
\end{proof}

For algebraically closed fields, Meshulam and \v Semrl have proved an even
better result:

\begin{theo}[Meshulam, \v Semrl \cite{MeshulamSemrlPAMS}]\label{MeshulamSemrlhyperplane}
Assume that $\K$ is algebraically closed, and let $\calS$ be an $n$-dimensional non-reflexive subspace of $\calL(U,V)$.
Then, $\mrk \calS \leq n$.
\end{theo}

It is noteworthy to point out that Meshulam and \v Semrl's method actually delivers a theorem that
is, at the core, a statement on LLD spaces:

\begin{theo}\label{LLDhyperplane}
Let $\calT$ be a hyperplane of an $(n+1)$-dimensional LLD space $\calS$.
Assume that $\K$ is algebraically closed. Then, $\mrk \calT \leq n$.
\end{theo}

In this last theorem, the upper bound $n$ is optimal. To see this, one considers
an arbitrary $(n+1)$-dimensional vector space $U$, and the space $\calS$ of all operators
from $U$ to $U \wedge U$ of the form $y \mapsto x \wedge y$ for some $x \in U$.
One checks that $\calS$ is $(n+1)$-dimensional LLD space and that every non-zero operator in $\calS$
has rank $n$. Thus, every hyperplane of $\calS$ has minimal rank $n$.
However, one can also prove that every hyperplane of $\calS$ is reflexive, so this
example does not tell us anything on the optimality of the upper bound in Theorem \ref{MeshulamSemrlhyperplane}.
In \cite{MeshulamSemrlPAMS}, Meshulam and \v Semrl prove that the optimal upper bound on the minimal rank of a
non-reflexive $n$-dimensional operator space must be greater than or equal to
$\lfloor \frac{n}{2}-\sqrt{n}\rfloor$.

The following proof of Theorem \ref{LLDhyperplane} is largely similar to the one of Meshulam and \v Semrl, with the notable exception of
the middle section, which does not use the ``differentiation trick" that is found in the original proof,
but only points to earlier known results on matrix spaces.

\begin{proof}[Proof of Theorem \ref{LLDhyperplane}]
Let $\varphi$ be a rank-optimal element of $\widehat{\calS}$, and set $r:=\rk \varphi$.
Using the line of reasoning from the proof of Theorem \ref{minranktheo1}, we find that
$\mrk \calT \leq n$ if $\Ker \varphi \cap \calT \neq \{0\}$, which holds true whenever $r<n$.
Now, we assume that $r=n$, we write $\Ker \varphi=\K g$, and we assume that $\calS=\calT \oplus \Ker \varphi$.
Denote by $\calM$ the matrix space associated with $\widehat{\calS}$ in respective bases of $\calS$ and $V$,
the first of which is adapted to the decomposition $\calS=\calT \oplus \K g$, and with the first $n$ vectors
of the second one forming a basis of $\im \varphi$.
Then, by Lemma \ref{Flanders}, some generic matrix of $\calM$ has the form
$$\mathbf{M}=\begin{bmatrix}
\mathbf{A} & \mathbf{X} \\
\mathbf{B} & [0]
\end{bmatrix},$$
where $\mathbf{A}$ and $\mathbf{X}$ are, respectively, $n \times n$ and $n \times 1$ matrices,
with $\rk \mathbf{A}=n$.
We write $\K[\mathbf{x_1},\dots,\mathbf{x_s}]$ for the polynomial ring used to construct $\mathbf{M}$.
As $\rk \mathbf{M}=n$, we find that
the kernel of $\mathbf{M}$ is spanned by the vector
$$\mathbf{Y}:=\begin{bmatrix}
\mathbf{A}^{-1}\mathbf{X} \\
-1
\end{bmatrix},$$
the entries of which are $0$-homogeneous fractions of polynomials.
Thus, $\Ker \mathbf{M}$ contains a non-zero vector of the form
$$\mathbf{Z}=\begin{bmatrix}\mathbf{p_1} & \cdots & \mathbf{p_n} & \mathbf{p_{n+1}} \end{bmatrix}^T,$$
where $\mathbf{p_1},\dots,\mathbf{p_{n+1}}$ are homogeneous polynomials with the same degree
$d$ and $\gcd(\mathbf{p_1},\dots,\mathbf{p_n},\mathbf{p_{n+1}})=1$.

According to Lemma \ref{basiclemma}, if by specializing $\mathbf{A}$ and $\mathbf{Z}$ at some point of $\K^s$, we find a
rank $n$ matrix $A$ and a vector $Z$, then $\dim \calM Z \leq n$.
Thus, using an algebraic density argument, the last result must hold for any specialization of $\mathbf{Z}$. Indeed:
\begin{enumerate}[(i)]
\item The set of all $Y \in \K^{n+1}$ such that $\dim \calM Y \leq n$ is linearly isomorphic, through duality,
to the one of all operators in $\calS$ with rank less than or equal to $n$, and that one is Zariski-closed.
\item The set of all $(x_1,\dots,x_s) \in \K^s$ at which the specialization of $\mathbf{A}$ is
invertible is Zariski-open in $\K^s$.
\end{enumerate}

In order to conclude, it suffices to show that some specialization of $\mathbf{Z}$ is a non-zero vector of
$\K^n \times \{0\}$, as such a vector would yield a non-zero operator $f \in \calT$ with $\rk f \leq n$.
Assume that such a vector does not exist. In other words, every zero of $\mathbf{p_{n+1}}$
is a zero of all the polynomials $\mathbf{p_1},\dots,\mathbf{p_n}$;
thus, by Hilbert's Nullstellensatz, $\mathbf{p_{n+1}}$ divides some power of each polynomial
$\mathbf{p_1},\dots,\mathbf{p_n}$, and the assumption that $\gcd(\mathbf{p_1},\dots,\mathbf{p_n},\mathbf{p_{n+1}})=1$
yields that $\mathbf{p_{n+1}}$ is constant. Therefore, $d=0$ and $\mathbf{Z}$ is a constant vector.
By specializing, this yields a non-zero vector $x \in \K^{n+1}$ on which all the matrices of
$\calM$ vanish, meaning that some non-trivial linear combination of the chosen basis of $\calS$ is zero. This contradiction concludes the proof.
\end{proof}

\subsection{Disproving a conjecture of Meshulam and \v Semrl}\label{disproveconjecturesection}

In \cite{MeshulamSemrlPAMS}, Meshulam and \v Semrl went on to conjecture that the upper
bound $n$ in Theorem \ref{LLDhyperplane} could hold for all large enough fields,
not only for algebraically closed ones.
Here, we disprove this through a very simple counterexample.

Let $n \geq 4$ be an even integer.
Let us assume for the moment that there exists a linear subspace $\calV$ of $\Mata_n(\K)$
with dimension $n-1$ in which all the non-zero matrices are non-singular (in which case we say that
$\calV$ is \textbf{non-singular}).
We consider orthogonality with respect to the symmetric bilinear form
$(X,Y) \mapsto X^TY$ on $\K^n$.
Note then that, for every non-zero vector $X \in \K^n$, one has
$\dim \calV X =n-1$ for, if not, we would find a non-zero element $A \in \calV$ with $AX=0$.
Thus $\calV X=\Mata_n(\K)X=\{X\}^\bot$ for all such $X$, as the inclusions $\calV X \subset \Mata_n(\K)X \subset \{X\}^\bot$
are obvious and $\dim \calV X=n-1=\dim \{X\}^\bot$.
Considering $\calV$ as a linear subspace of the endomorphism space $\calL(\K^n)$,
it follows that the reflexive closure of $\calV$ contains $\Mata_n(\K)$, and one can even check
that this reflexive closure equals $\Mata_n(\K)$. As $\dim \Mata_n(\K)=\dbinom{n}{2}>n-1$, this shows in particular that
$\calV$ is non-reflexive. However, we have assumed that every non-zero element of $\calV$ has rank $n$.
Thus, we have an $(n-1)$-dimensional non-reflexive space of operators in which all the non-zero operators have rank greater than $n-1$.

It remains to give examples of such spaces $\calV$.
Applying the Chevalley-Warning theorem to the pfaffian, one sees that no such subspace exists for a finite field
(see \cite[Satz 1]{Heineken} and \cite[Lemma 3]{GowQuinlan}).

For the case $n=4$, a $3$-dimensional linear subspace $\calV$ of $\Mata_4(\K)$
is non-singular if and only if the pfaffian does not vanish on $\calV \setminus \{0\}$.
Thus, we have the
example of the subspace of $\Mata_4(\R)$ defined by the generic matrix
$$\begin{bmatrix}
0 & -\mathbf{b} & \mathbf{c} & -\mathbf{d} \\
\mathbf{b} & 0 & \mathbf{d} & \mathbf{c} \\
-\mathbf{c} & -\mathbf{d} & 0 & \mathbf{b} \\
\mathbf{d} & -\mathbf{c} & -\mathbf{b} & 0
\end{bmatrix},$$
the pfaffian of which is $-\mathbf{b}^2-\mathbf{c}^2-\mathbf{d}^2$.
Note the connection with the standard representation of pure quaternions.

The division algebra of octonions yields an example of a non-singular subspace
$\calV$ of $\Mata_8(\R)$ with dimension $7$: the one defined by the generic matrix
$$\begin{bmatrix}
0 & -\mathbf{b} & -\mathbf{c} & -\mathbf{d} & -\mathbf{e} & -\mathbf{f} & -\mathbf{g} & -\mathbf{h} \\
\mathbf{b} & 0 & -\mathbf{d} & \mathbf{c} & \mathbf{f} & -\mathbf{e} & \mathbf{h} & -\mathbf{g} \\
\mathbf{c} & \mathbf{d} & 0 & -\mathbf{b} & \mathbf{g} & -\mathbf{h} & -\mathbf{e} & \mathbf{f} \\
\mathbf{d} & -\mathbf{c} & \mathbf{b} & 0 & \mathbf{h} & \mathbf{g} & -\mathbf{f} & -\mathbf{e} \\
\mathbf{e} & -\mathbf{f} & -\mathbf{g} & -\mathbf{h} & 0 & \mathbf{b} & \mathbf{c} & \mathbf{d} \\
\mathbf{f} & \mathbf{e} & \mathbf{h} & -\mathbf{g} & -\mathbf{b} & 0 & \mathbf{d} & -\mathbf{c} \\
\mathbf{g} & -\mathbf{h} & \mathbf{e} & \mathbf{f} & -\mathbf{c} & -\mathbf{d} & 0 & \mathbf{b} \\
\mathbf{h} & \mathbf{g} & -\mathbf{f} & \mathbf{e} & -\mathbf{d} & \mathbf{c} & -\mathbf{b} & 0
\end{bmatrix}.$$

In general, if we have a non-singular linear subspace $\calV$ of $\Mata_n(\K)$ with dimension $n-1$
together with a non-isotropic matrix $P \in \Mat_n(\K)$, i.e.\ a matrix such that the quadratic form
$X \mapsto X^TPX$ is non-isotropic, then $\K P \oplus \calV$ is a linear subspace of $\Mat_n(\K)$
in which every non-zero matrix is invertible. Indeed, we already know that every non-zero matrix of $\calV$
is invertible while, for every $\lambda \in \K \setminus \{0\}$ and every $M \in \calV$,
the matrix $\lambda P+M$ is invertible since the associated quadratic form $X \mapsto X^T(\lambda P+M)X=\lambda\, X^TPX$
is non-isotropic.
However, we know - see e.g \cite[Section 4]{dSPlinpresGL} -
that the $n$-dimensional linear subspaces of $\Mat_n(\K)$ in which all the non-zero elements are
non-singular correspond to the regular bilinear maps from $\K^n \times \K^n$ to $\K^n$,
i.e.\ to the structures of division algebra on the vector space $\K^n$.

For $\K=\R$, we can take $P=I_n$, and hence the Bott-Kervaire-Milnor theorem
\cite{BottMilnor,Kervaire} on real division algebras
shows that a non-singular linear subspace $\calV$ of $\Mata_n(\R)$ with dimension $n-1$ exists only if $n \in \{1,2,4,8\}$.

\subsection{The case of small finite fields}

In Section \ref{smallranksection}, we have recalled a theorem of Meshulam and \v Semrl that states that
every $n$-dimensional LLD operator space contains a non-zero operator of rank less than $n$.
We have seen that this result holds for all fields, although a different argument is needed for small finite fields.

What can be said of the minimal rank in a hyperplane of an $(n+1)$-dimensional LLD space when the underlying field has less than $n+1$ elements?
Although the upper bound $2n-2$ of Theorem \ref{minranktheo1} probably looks like a good candidate, it is currently out of our reach.
On the other hand, Ding \cite{Ding} has proved that every $n$-dimensional non-reflexive operator space has
minimal rank at most $n^2$, whatever the underlying field.
In this section, we shall improve Ding's result as follows:

\begin{theo}\label{minranktheo2}
Let $\calT$ be a linear subspace of an $(n+1)$-dimensional LLD operator space $\calS$, with $n>0$
and $\dim \calT \geq n$.
Then, $\mrk \calT \leq \dbinom{n+1}{2}$.
\end{theo}

Note how this also improves Amitsur's lemma (see Section \ref{Amitsurnote}). Unsurprisingly,
our proof will borrow from Amitsur's method (see his proof of Lemma 1 in \cite{Amitsur}).

\begin{proof}
We prove the result by induction on $n$. For $n=1$, we know that the essential range of
$\calS$ has dimension $1$, and hence $\rk f =1$ for every non-zero operator $f \in \calS$.
Assume now that $n>1$. We may assume that $\calS$ is minimal among LLD spaces: indeed,
if the opposite is true, then we may find an LLD subspace $\calS'$ of $\calS$ with $2 \leq \dim \calS'<\dim \calS$,
and in any case $\dim (\calT \cap \calS') \geq \dim \calS'-1$, so the result follows by induction.

As $\calS$ is minimal, it is not $2$-LLD and hence $\widehat{\calS}$ has upper-rank $n$. Let $\varphi$ be a rank-optimal element of
$\widehat{\calS}$, set $V_0:=\im \varphi$ and choose a non-zero operator $g \in \Ker \varphi$.
Let us denote by $\pi : V \twoheadrightarrow V/V_0$ the canonical projection, and set
$$\calC:=\{f \in \calS : \; \pi \circ f=0\}.$$

If $\dim \calC>1$, then $\calC \cap \calT$ contains a non-zero element $f$, so that
$\im f \subset V_0$; in that case, we have found an element of $\calT$ with rank at most $n \leq \dbinom{n+1}{2}$.

Now, we assume that $\dim \calC \leq 1$ and that $\calC$ does not contain $g$.
Then, we can choose a hyperplane $\calU$ of $\calS$ such that $\calC \cap \calU=\{0\}$ and $g \in \calU$.
Thus, $\calS':=\{\pi \circ f \mid f \in \calU\}$ has dimension $n$.
Following Amitsur, we prove that $\calS'$ is LLD. Indeed, assume on the contrary that some $x \in U$ is such that
$\pi(f(x)) \neq 0$ for all non-zero operators $f \in \calU$. Then, $\dim \{f \in \calS : \; \pi(f(x))=0\} \leq 1$, and, as
$\calS$ is LLD, we deduce that $\{f \in \calS : \; \pi(f(x))=0\}=\{f \in \calS : \; f(x)=0\}$.
However, writing $\varphi : f \mapsto f(y)$ for some $y \in U$, we note that
$\pi(f(x+y))=\pi(f(x))$ for all $f \in \calS$. Applying the above results to both vectors $x$ and $x+y$, we deduce that
$\{f \in \calS : \; f(x)=0\}=\{f \in \calS : \; f(x+y)=0\}$, whence
$\forall f \in \calS, \; f(x)=0 \Rightarrow f(y)=0$. Since $\calS$ is LLD, we may choose a non-zero operator $f \in \calS$ with
$f(x)=0$. Then, $f \in \Ker \varphi$, and hence $f \in \calU$. But $\pi(f(x))=0$ contradicts the above
results on $\calU$. Therefore, $\calS'$ is LLD.

From there, we note that $\dim \bigl((\pi \calT) \cap \calS'\bigr) \geq \dim \calS'-1$, which, by induction,
yields a non-zero operator $f \in \calT$
such that $\rk (\pi \circ f) \leq \dbinom{n}{2}$. Therefore,
$$\rk f \leq \dim V_0+\rk (\pi \circ f) \leq n+\binom{n}{2}=\binom{n+1}{2}.$$
It remains to consider the case when $\calC=\K g=\Ker \varphi$ (note that the Flanders-Atkinson lemma shows that
this is always the case if $\# \K > n$). In that situation, we note that $\im g \subset V_0$.

If $g \in \calT$, we deduce that $\mrk \calT \leq \dim V_0=n$.

Let us assume further that $g \not\in \calT$, so that $\calT':=\{\pi \circ f \mid f \in \calT\}=\{\pi \circ f \mid f \in \calS\}$ has dimension $n$.
If $\calT'$ is LLD, the above line of reasoning yields $\mrk \calT \leq \dbinom{n+1}{2}$, once more.
Assume further that $\calT'$ is not LLD, and choose $x \in U$ such that $\pi(f(x)) \neq 0$ for all non-zero operators $f \in \calT$.
Then, for $\psi : f \in \calS \mapsto f(x)$, we have $\dim \Ker \psi \leq 1$, and hence $\psi$ is rank-optimal in $\widehat{\calS}$.
Note that $\{f \in \calS : \; \pi(f(x))=0\}$ contains $g$ and has dimension at most $1$.
Hence, $\{f \in \calS : \; \pi(f(x))=0\}=\K g$. It follows that $\Ker \psi \subset \K g$, and,
as $\Ker \psi \neq \{0\}$, one deduces that
$\Ker \psi=\Ker \varphi$. Setting $V_1:=\im \psi$, we see that $V_0 \cap V_1 =\{0\}$ for the opposite would yield
$\dim (V_0+V_1)/V_0 <\dim V_1=n$, contradicting the fact that $f \in \calT \mapsto \pi(f(x))\in (V_0+V_1)/V_0$ is linear and one-to-one.
Finally, we can conclude. Replacing $\varphi$ with $\psi$ in the first part of our proof, we see that either
$\mrk \calT \leq \dbinom{n+1}{2}$ or $\im g \subset V_1$. In the second case, we would find $\im g \subset V_0 \cap V_1=\{0\}$,
contradicting the fact that $g \neq 0$.

Therefore, in any case we have shown that $\mrk \calT \leq \dbinom{n+1}{2}$, which completes the proof by induction.
\end{proof}

\section{The minimal rank problem in a non-reflexive operator space (II)}\label{nonreflexivesectionII}

Now that the Meshulam-\v Semrl conjecture has been disproved,
the question remains whether the upper bound $2n-2$ in Corollary \ref{minrankcor1}
is the best one available in general or if one can come up with a lower one.
In this section, we give a general construction of examples in which the upper bound $2n-2$ is attained.

Quadratic forms play a large part in the following considerations, and the reader should be aware that a
solid knowledge of the theory of quadratic forms over fields is necessary beyond this point. In particular,
Witt's theory will be needed in the last part of the section.
Nevertheless, we shall recall some basic facts here as they are sufficient to grasp the main results.

Over an arbitrary field (whatever its characteristic), a quadratic form on a finite-dimensional vector space
$V$ is a map of the form $q : x \mapsto b(x,x)$, where $b : V \times V \rightarrow \K$ is a bilinear form
(but $b$ can be non-symmetric). The polar form of $q$ is then defined as the symmetric
bilinear form $b_q : (x,y) \mapsto q(x+y)-q(x)-q(y)=b(x,y)+b(y,x)$, and orthogonality refers to this bilinear form
when one speaks of orthogonality with respect to $q$.
Note that if $b$ is symmetric, then the polar form of $q$ is $2\,b$ (and thus it is zero if in addition $\K$ has characteristic $2$).
Note that, for fields of characteristic not $2$, the polar form of $q$ is usually defined as $(x,y) \mapsto \frac{1}{2}\,(q(x+y)-q(x)-q(y))$
but the need for a unified treatment motivates the above definition. Over a field of characteristic $2$, the polar
form of $q$ is always alternating, in the sense that $b_q(x,x)=0$ for all $x \in V$.

The radical of the quadratic form $q$ is defined as $\{x \in V : \; b_q(x,-)=0\}$, and $q$ is called non-degenerate (or regular) when
its radical equals $\{0\}$. The form $q$ is called isotropic when there is a non-zero vector $x \in V \setminus \{0\}$ such that $q(x)=0$.
Over a field of characteristic not $2$, a non-isotropic quadratic form is always non-degenerate
(this is not true however over fields of characteristic $2$; for example the quadratic form $x \mapsto x^2$
on such a field is non-isotropic although its polar form is zero).

Two quadratic forms $q$ and $q'$ (respectively on $V$ and $V'$) are called equivalent when
there is a linear isomorphism $u : V \overset{\simeq}{\rightarrow} V'$ such that $\forall x \in V, \; q(x)=q'(u(x))$;
in this case, we write $q \simeq q'$ and we say that $u$ is an isometry from $q$ to $q'$.
The orthogonal group of the quadratic space $(V,q)$ is the group of all isometries from $q$ to itself.

Given quadratic forms $q_1$ and $q_2$ on, respectively, spaces $V_1$ and $V_2$, we denote by
$q_1 \bot q_2$ their orthogonal direct sum as defined by $(q_1 \bot q_2) : (x,y)\in V_1 \times V_2 \mapsto q_1(x)+q_2(y)$.
Given scalars $a_1,\dots,a_n$, one denotes by $\langle a_1,\dots,a_n\rangle$ the
quadratic form $(x_1,\dots,x_n) \mapsto \underset{k=1}{\overset{n}{\sum}} a_k x_k^2$ on $\K^n$.
A quadratic form $q$ is called hyperbolic when it is equivalent to the orthogonal direct sum of finitely many copies of
the standard hyperbolic form $(x,y) \mapsto xy$ on $\K^2$.

We shall not recall the general definition of the tensor product of quadratic forms, but we simply remind the reader that
$\langle a_1,\dots,a_n \rangle \otimes q \simeq (a_1\,q)\bot (a_2\,q)\bot \cdots \bot (a_n\,q)$ whenever $q$ is a quadratic form and
$a_1,\dots,a_n$ are scalars.

Assume for now that $\K$ has characteristic $2$: given $(a,b) \in \K^2$, one denotes by $[a,b]$
the quadratic form $(x,y) \mapsto ax^2+xy+by^2$ on $\K^2$. In particular, every $2$-dimensional hyperbolic form is
equivalent to $[0,0]$. The map $\calP : x \mapsto x^2+x$ is an endomorphism of the group $(\K,+)$.
If $q$ is a regular quadratic form with dimension $2n$, then, in some basis, one of the bilinear forms $b$ that satisfies
$\forall x \in V, \; q(x)=b(x,x)$ is represented by
$\begin{bmatrix}
A & [?]_{n \times n} \\
[0]_{n \times n} & B
\end{bmatrix}$, where $A$ and $B$ belong to $\Mat_n(\K)$, and the class of $\tr(AB)$ in the quotient group
$\K/\calP(\K)$ depends only on $q$ and is called the \emph{Arf invariant} of $q$. In particular, the Arf invariant of
$[a,b]$ is the class of $ab$ mod. $\calP(\K)$ (see \cite[Chapter XXXII Section 4]{invitquad}).

\subsection{LDB division algebras}\label{LDBdef}

\begin{Def}
Let $A$ be a vector space. A bilinear map $\star : A \times A \rightarrow A$
is called \textbf{regular} when $a \star b \neq 0$ for all $(a,b)\in (A \setminus \{0\})^2$.
\end{Def}

\begin{Def}
Let $(A,\star)$ be a finite-dimensional division algebra, i.e.\ $A$ is a finite-dimensional vector space over $\K$ and $\star$ is
a regular bilinear mapping from $A \times A$ to $A$.
A \textbf{quasi-left-inversion of $\star$} is a binary operation $\bullet : A \times A \rightarrow A$
which vanishes exactly on $(A \times \{0\}) \cup (\{0\} \times A)$ and for which
the vectors $x \star (x \bullet y)$ and $y$ are colinear for all $(x,y) \in A^2$.

A \textbf{left-division-bilinearizable} (in abbreviated form: \textbf{LDB}) division algebra is a triple
$(A,\star,\bullet)$, where $(A,\star)$ is a finite-dimensional division algebra, \emph{with positive dimension},
and $\bullet$ is a bilinear quasi-left-inversion of $\star$.
\end{Def}

Here is the first important result on LDB division algebras:

\begin{prop}
Let $(A,\star,\bullet)$ be an LDB division algebra.
Then, there exists a unique quadratic form $q$ on $A$ such that
$$\forall (x,y) \in A^2, \; x \star (x \bullet y)=q(x)\,y.$$
Moreover, $q$ is non-isotropic. We shall say that $q$ is the \textbf{quadratic form attached to $(A,\star,\bullet)$}.
\end{prop}

\begin{proof}
Let $x \in A$. Then, $y \mapsto x \star (x \bullet y)$ is linear and maps every vector of $A$ to a collinear vector.
Therefore, there is a unique scalar $q(x)\in \K$ such that
$$\forall y \in A, \; x \star (x \bullet y)=q(x)\,y.$$
Let us choose a non-zero vector $y \in A$ and a linear form $\varphi \in A^\star$ such that $\varphi(y)=1$.
Then, $\forall x \in A, \; q(x)=\varphi(x \star (x \bullet y))$, and as
$$(x_1,x_2) \mapsto \varphi(x_1 \star (x_2 \bullet y))$$
is a bilinear form on $A$ we deduce that $q$ is a quadratic form on $A$.
Finally, if $x \neq 0$ then $q(x)\, y \neq 0$ as $\star$ and $\bullet$ are regular, whence $q(x)\neq 0$.
\end{proof}

It follows that for a finite field $\K$, an LDB division algebra over $\K$ has dimension at most $2$.

If $\K$ has characteristic $2$, then one should be aware that the quadratic form $q$ may be degenerate, in the
sense that its (alternating) polar form $(x,y) \mapsto q(x+y)-q(x)-q(y)$ may be degenerate, even when
the dimension of $A$ is even. An example of this will soon be given.

\begin{Rem}\label{inversionremark}
Let $(A,\star,\bullet)$ be an LDB division algebra with attached quadratic form $q$.
Then, $(A,\bullet,\star)$ is an LDB division algebra with attached quadratic form $q$!
Indeed, let $(x,y) \in A^2$ be with $x \neq 0$.
Then,
$$x \star (x \bullet (x \star y))=q(x)\,(x \star y)=x \star (q(x)\, y),$$
and since $x$ is non-zero and $\star$ is regular, this yields
$$x \bullet (x \star y)=q(x)\,y.$$
\end{Rem}

Now, say that $A=\K^n$ with its standard vector space structure.
Using canonical matrix representations, a regular
bilinear map $\star : A^2 \rightarrow A$ can be seen as a one-to-one linear map $M : a \in \K^n \mapsto (a \star -) \in \Mat_n(\K)$
which maps every non-vector vector $a$ to an invertible matrix.
A bilinear quasi-left-inversion of $\star$ can then be viewed as a one-to-one linear map
$N : \K^n \hookrightarrow \Mat_n(\K)$ for which there exists a non-isotropic quadratic form $q$ in
$\K^n$ such that
$$\forall a \in \K^n, \; N(a)M(a)=q(a)\,I_n$$
i.e.\
$$\forall a \in \K^n \setminus \{0\}, \; N(a)=q(a)\,M(a)^{-1}.$$
For $n=2$, a basic example is obtained by taking a $2$-dimensional linear subspace $\calV$ of $\Mat_2(\K)$
in which every non-zero matrix is invertible; then one takes an arbitrary isomorphism $i : \K^2 \overset{\simeq}{\rightarrow} \calV$.
An obvious bilinear quasi-left-inversion of $i$ is $x \mapsto \widetilde{i(x)}$, where $\widetilde{M}$ denotes the
transpose of the comatrix of $M$, the associated quadratic form being $x \mapsto \det i(x)$.
This shows in particular that every $2$-dimensional division algebra has a bilinear quasi-left-inversion.
Recall that every $2$-dimensional division algebra is equivalent to a quadratic extension of $\K$, and
two quadratic extensions of $\K$ are equivalent as division algebras if and only if they are isomorphic
as field extensions of $\K$. Note also that if $\L$ is a quadratic extension of $\K$, then:
\begin{itemize}
\item Either $\L$ is separable over $\K$, and then $(x,y) \mapsto \sigma(x)y$ is
a bilinear quasi-left-inversion of the product on $\L$, where $\sigma$ is the non-identity automorphism of the $\K$-algebra $\L$;
the attached quadratic form is the norm of $\L$ over $\K$.
\item Or $\L$ is inseparable over $\K$, and $\K$ has characteristic $2$;
then, the product on $\L$ is a bilinear quasi-left-inversion of itself and the attached quadratic form is $x \in \L \mapsto x^2 \in \K$.
In that case, the attached quadratic form is totally degenerate, i.e.\ its polar form is zero!
\end{itemize}

\vskip 3mm
Another set of examples is yielded by quaternion algebras.
Let $q$ be a $2$-dimensional quadratic form over $\K$ such that $\langle 1\rangle \bot (-q)$ is non-isotropic.
Then, we consider the Clifford algebra $C(q)$ associated with $q$. Recall that $C(q)$ is then a skew-field extension of $\K$
of degree $4$, on which there is an anti-automorphism of $\K$-algebra $x \mapsto x^\star$
called the \textbf{conjugation}, together with a quadratic form $N$ on $C(q)$ called the \textbf{norm} of the quaternion algebra $C(q)$ and which satisfies
$\forall x \in C(q), \; x^\star x=xx^\star=N(x)$.
Thus, $(x,y) \mapsto x^\star y$ is a bilinear quasi-left-inversion of the product of $C(q)$ and the attached quadratic form is $N$.
Note that if $\K$ has characteristic not $2$, then $N \simeq \langle 1,\delta\rangle \bot (-q)$,
where $\delta$ is a determinant of $q$; if $\K$ has characteristic $2$, then
either $q \simeq \langle a,b\rangle$ for some $(a,b)\in (\K^*)^2$ and then
$N \simeq \langle 1,a\rangle \otimes \langle 1,b\rangle$, or
$q$ is non-degenerate and then $N \simeq [1,\delta] \bot q$ where $\delta$ represents the Arf invariant of $q$.

In particular, for the field of real numbers and $q=\langle -1,-1\rangle$, this construction yields
the standard skew field of quaternions, and a relevant
$4$-dimensional subspace of $\Mat_4(\K)$ is given by the generic matrix
$$\begin{bmatrix}
\mathbf{a} & -\mathbf{b} & -\mathbf{c} & -\mathbf{d} \\
\mathbf{b} & \mathbf{a} & \mathbf{d} & -\mathbf{c} \\
\mathbf{c} & -\mathbf{d} & \mathbf{a} & \mathbf{b} \\
\mathbf{d} & \mathbf{c} & -\mathbf{b} & \mathbf{a}
\end{bmatrix}.$$

Finally, we can find larger LDB division algebras by using Cayley's generalized octonions:
starting from the above quaternion algebra $C(q)$, one chooses, if possible, a scalar $\varepsilon$ such that
$\langle 1,-\varepsilon\rangle \otimes N$ is non-isotropic (note that this is equivalent to having
$\varepsilon$ outside the range of $N$, since $N$ is multiplicative), and one uses the Cayley-Dickson construction to define
an inner composition law $*$
on $C(q)^2$ by
$$(a,b) * (c,d):=(ac-db^\star\, ,\, a^\star d-\varepsilon cb).$$
One checks that this endows $C(q)^2$ with a structure of $8$-dimensional division algebra over $\K$,
and the pairing defined by
$$(a,b) \bullet (c,d):=(a^\star c+db^\star \, , \, ad+\varepsilon cb)$$
provides a bilinear quasi-left-inversion of $*$ with attached quadratic form
$N \bot (-\varepsilon\,N) \simeq \langle 1,-\varepsilon\rangle \otimes N$.

\vskip 3mm
Let us resume the general theory of LDB division algebras.
Note that if $\bullet$ is a bilinear quasi-left-inversion of $\star$, then $\lambda\,\bullet$ is obviously another one for each $\lambda \in \K \setminus \{0\}$.
We prove that this yields all the possible bilinear quasi-left-inversion maps:

\begin{prop}\label{unicitedelinversion}
Let $\circ$ and $\bullet$ be two bilinear quasi-left-inversion maps for the division algebra $(A,\star)$.
Then, $\bullet$ is a scalar multiple of $\circ$.
\end{prop}

\begin{proof}
The result is obvious if $A$ has dimension at most $1$, for in that case all the bilinear maps from $A^2$ to $A$
are collinear. Assume now that $\dim A>1$.
Then, for all
$x \in A \setminus \{0\}$, the endomorphisms $x \bullet -$ and $x \circ -$ are collinear vectors of $\calL(A)$
as they are both scalar multiples of $(x \star -)^{-1}$.

It follows that the maps $f : x \mapsto x \bullet -$ and $g : x \mapsto x \circ -$
are locally linearly dependent. However, $\rk g =\dim A \geq 2$ as $\circ$ is regular, and hence
$f$ is a scalar multiple of $g$, which yields the claimed result.
\end{proof}

It follows that the quadratic form attached to an LDB division algebra $(A,\star,\bullet)$
is uniquely determined, up to multiplication by a non-zero scalar, by $\star$.

Interestingly, LDB division algebras are connected with the examples we have discussed in Section \ref{disproveconjecturesection}:

\begin{prop}
Let $(A,\star,\bullet)$ be an $n$-dimensional LDB division algebra, with $n \geq 2$. Then,
there exists an $(n-1)$-dimensional linear subspace of $\Mata_n(\K)$ in which all the non-zero matrices are invertible,
and $n$ is even.
\end{prop}

\begin{proof}
Without loss of generality, we may assume that $A=\K^n$.
Denote by $q$ the quadratic form attached to $(\K^n,\star,\bullet)$.
For $X \in \K^n$, let us denote, respectively, by
$A(X)$ and $B(X)$ the matrices of $X \star -$ and $X \bullet -$ in the canonical basis.
Then, $X \mapsto A(X)$ and $X \mapsto B(X)$ are one-to-one linear maps from $\K^n$ to $\Mat_n(\K)$, and we have
\begin{equation}\label{matrixidentity}
\forall X \in \K^n, \; A(X)B(X)=q(X)I_n.
\end{equation}
We may find matrices $A_1,\dots,A_n$ and $B_1,\dots,B_n$, all in $\Mat_n(\K)$, such that
$$\forall X \in \K^n, \; A(X)=\begin{bmatrix}
X^T A_1 \\
\vdots \\
X^T A_n
\end{bmatrix} \quad \text{and} \quad
B(X)=\begin{bmatrix}
B_1X & \cdots & B_n X
\end{bmatrix}.$$
As $A(X)$ and $B(X)$ are non-singular for all non-zero vectors $X \in \K^n$, each non-trivial linear combination of $A_1,\dots,A_n$
is non-singular, and so is each non-trivial linear combination of $B_1,\dots,B_n$.
From \eqref{matrixidentity}, we find
$$\forall X \in \K^n, \;\forall i \in \lcro 2,n\rcro, \; X^T A_1B_iX=0,$$
and hence every matrix of $\calV:=A_1\Vect(B_2,\dots,B_n)$ is alternating.
On the other hand, every non-zero matrix of $\calV$ is non-singular, and $\dim \calV=n$.
One concludes by noting that $\Mata_n(\K)$ contains a non-singular matrix only if $n$ is even.
\end{proof}

By way of consequence, if there is an $n$-dimensional LDB division algebra over $\K$, with $n \geq 4$, then
there is also an $(n-1)$-dimensional non-reflexive operator space in which all the non-zero operators have rank $n$
(see Section \ref{disproveconjecturesection}).

\begin{Rem}
Using the arguments of the proof, one sees that defining an $n$-dimensional LDB division algebra structure on the $\K$-vector space
$\K^n$ amounts to finding a list $(A_1,\dots,A_n) \in \Mat_n(\K)^n$ for which there is
a non-isotropic matrix $P \in \GL_n(\K)$ together with a list $(B_1,\dots,B_n) \in \Mat_n(\K)^n$ satisfying
$$\forall (i,j)\in \lcro 1,n\rcro^2, \; A_iB_j-\delta_{i,j}P \in \Mata_n(\K),$$
where $\delta_{i,j}=1$ if $i=j$, and $\delta_{i,j}=0$ otherwise.
\end{Rem}

Here is an alternative proof that the dimension of an LDB division algebra is either $1$ or an even number.
Let $(A,\star,\bullet)$ be an $n$-dimensional LDB division algebra with attached quadratic form denoted by $q$.
Assume that $n \geq 3$. Then, as $q$ is non-isotropic the field $\K$ must be infinite.
For all $x \in A$, we have $(x \star -) \circ (x \bullet -)=q(x)\,\id_A$ and hence
$$\det(x \star -)\det(x \bullet -)=q(x)^n.$$
Note that $x \mapsto \det(x \star -)$ and $x \mapsto \det(x \bullet -)$ are both homogeneous polynomial functions of degree $n$
on the vector space $A$.
However, as $q$ is a non-isotropic quadratic form and $n \geq 2$, the polynomial map $q$
is irreducible. Thus, $x \mapsto \det(x \star -)$ is the product of a power of $q$
with a non-zero scalar; as its degree is $n$ and the one of $q$ is $2$, we deduce that $n$ is even.

\subsection{The twisted operator space attached to an LDB division algebra}\label{twisteddef}

Let $(A,\star,\bullet)$ be an LDB division algebra.
The bilinear mapping
$$\Gamma_{A,\star,\bullet} : \begin{cases}
(A \oplus \K^2) \times A^2 & \longrightarrow A^2 \\
\bigl(x +(\lambda,\mu),(y,z)\bigr) & \longmapsto (x \star z+\lambda\, y\, ,\, x \bullet y+\mu\, z)
\end{cases}$$
is left-regular.
Indeed, given $(x,\lambda,\mu) \in A \times \K^2$ such that $\Gamma_{A,\star,\bullet}(x+(\lambda,\mu),-)=0$,
we find $\lambda\,y=0$ and $x \bullet y=0$ for all $y \in A$ by taking the vector $(y,0)$,
and similarly we find $\mu z=0$ for all $z \in A$; this yields $\lambda=0$, $x=0$ and $\mu=0$.

\begin{Def}
The \textbf{twisted operator space} $\calT_{A,\star,\bullet}$ attached to the LDB division algebra $(A,\star,\bullet)$
is the vector space of all endomorphisms
$\Gamma_{A,\star,\bullet}(x+(\lambda,\mu),-)$ of $A^2$, for $(x,\lambda,\mu) \in A \times \K^2$.
\end{Def}

As is customary, we shall simply write $\calT_A$ (respectively, $\Gamma_A$) instead of $\calT_{A,\star,\bullet}$ (respectively, of
$\Gamma_{A,\star,\bullet}$) when no confusion can reasonably arise on the pair of laws $(\star,\bullet)$.
Note that $\calT_A$ has dimension $\dim A+2$. Moreover, $\calT_A$ is reduced because it contains $\id_{A^2}=\Gamma_A\bigl((1,1),-\bigr)$.

Obviously, $\calT_{A,\star,\bullet}$ depends on both laws $\star$ and $\bullet$.
However, given $\alpha \in \K \setminus \{0\}$, the operator spaces $\calT_{A,\star,\bullet}$
and $\calT_{A,\star,\alpha \bullet}$ are equivalent. Indeed, for the isomorphisms
$F : x+(\lambda,\mu) \in A \oplus \K^2 \longmapsto x+(\lambda,\alpha^{-1}\mu) \in A \oplus \K^2$ and
$G : (y,z)\in A^2 \mapsto (y,\alpha z) \in A^2$ one checks that
$$\forall (X,Y) \in (A \oplus \K^2) \times A^2, \quad
\Gamma_{A,\star,\alpha \bullet}(X,Y)=G\bigl(\Gamma_{A,\star,\bullet}(F(X),Y)\bigr).$$

\begin{prop}\label{LDBdonneLLD}
Let $(A,\star,\bullet)$ be an LDB division algebra.
Then, $\calT_A$ is LLD and more precisely $\{f \in \calT_A : \; f(y,z)=0\}$ has dimension $1$ for all
$(y,z) \in A^2 \setminus \{(0,0)\}$.
\end{prop}

\begin{proof}
Denote by $q$ the quadratic form attached to $(A,\star,\bullet)$.
Let $(y,z) \in A^2 \setminus \{(0,0)\}$ and $(x,\lambda,\mu) \in A \times \K^2$.
Then,
$$\Gamma_A\bigl(x+(\lambda,\mu),(y,z)\bigr)=0 \; \Leftrightarrow \; \begin{cases}
x \star z & =-\lambda\,y \\
x \bullet y & =-\mu\, z.
\end{cases}$$
If $z=0$, then $y \neq 0$ and the above condition is equivalent to $\lambda=0$ and $x=0$; in that case,
$\{f \in \calT_A : \; f(y,z)=0\}$ has dimension $1$. \\
If $y=0$, then a similar line of reasoning yields the same conclusion. \\
Assume finally that $y \neq 0$ and $z \neq 0$. In the above condition, the additional condition $x=0$
would lead to $\lambda=\mu=0$. Assume now that $x \neq 0$, and denote by
$x_0$ the sole vector of $A$ for which $x_0 \star z=y$. Then,
$$x \bullet y =-\mu\,z \, \Leftrightarrow\, x \star (x \bullet y)=x \star (-\mu z) \, \Leftrightarrow \, q(x)\,y=-\mu \, x\star z,$$
and hence the above set of conditions is equivalent to
$$x=-\lambda\,x_0 \quad \text{and} \quad \mu=q(x_0)\, \lambda,$$
which defines a $1$-dimensional subspace of $A \oplus \K^2$, as claimed.
\end{proof}

Using $\Gamma_A$ to identify $\calT_A$ with $A \oplus \K^2$, we may consider the quadratic form
$$\widetilde{q} : x+(\lambda,\mu) \mapsto q(x)-\lambda \mu$$
on $\calT_A$. Note that $\widetilde{q}$ is equivalent to $q \bot \langle 1,-1\rangle$ if $\K$ has characteristic not $2$,
and to $q \bot [0,0]$ otherwise.

Now, we show that the singular operators in $\calT_A$ are the zeros of the quadratic form
$\widetilde{q}$:

\begin{prop}\label{rankinatwisted}
Let $(A,\star,\bullet)$ be an $n$-dimensional LDB division algebra
with attached quadratic form $q$. Let $f \in \calT_A \setminus \{0\}$.
Then $\rk f=n$ or $\rk f=2n$, whether $f$ is a zero of the quadratic form $\widetilde{q}$ or not.
\end{prop}

\begin{proof}
Let $(x,\lambda,\mu) \in A \times \K^2$ be such that $f=\Gamma_A(x+(\lambda,\mu),-)$. \\
If $x=0$, $\lambda=0$ and $\mu \neq 0$, one obviously has $\rk f=n$.
The same holds if $x=0$, $\lambda \neq 0$ and $\mu=0$. \\
Assume now that $x \neq 0$.
Let $(y,z) \in A^2$. Then,
$$f(y,z)=0 \; \Leftrightarrow \; \begin{cases}
x\star z=-\lambda y \\
x \bullet y=-\mu z.
\end{cases}$$
As $x$ is non-zero, the second condition is equivalent to
$x \star (x \bullet y)=-\mu\, x \star z$, i.e. to $q(x)\, y=-\mu\, x \star z$. Therefore,
$$f(y,z)=0 \; \Leftrightarrow \; \begin{cases}
x\star z=-\lambda y \\
\bigl(q(x)-\lambda\mu\bigr)\,y=0.
\end{cases}$$
If $q(x) \neq \lambda \mu$, one deduces that $f$ is one-to-one.
Otherwise, the kernel of $f$ is the $n$-dimensional subspace
$\bigl\{(x \star t,-\lambda t)\mid t \in A\}$, and hence $\rk f=2n-n=n$.
\end{proof}

Thus, any non-isotropic hyperplane of $\calT_A$ yields a case when the upper bound
in Corollary \ref{minrankcor1} is attained for the integer $n+1$. Moreover, the following results
show that such a hyperplane is always non-reflexive, and, better still, that $\calT_A$ is its reflexive closure
whenever $\dim A \geq 3$.

\begin{prop}
Let $\calH$ be a non-isotropic hyperplane of $\calT_A$.
Then $\calT_A \subset \calR(\calH)$.
\end{prop}

\begin{proof}
Let $f \in \calT_A \setminus \calH$.
Let $(y,z) \in A^2$. As $\calT_A$ is LLD, some non-zero operator $g \in \calT_A$ vanishes at $(y,z)$.
However $g \not\in \calH$ as all the non-zero operators of $\calH$ are non-singular.
Therefore $g=\alpha\,f+h$ for some $\alpha \in \K \setminus \{0\}$ and some $h \in \calH$,
which yields $f(y,z)=-\frac{1}{\alpha}\,h(y,z)$. Therefore, $f \in \calR(\calH)$.
\end{proof}

\begin{prop}
Let $(A,\star,\bullet)$ be an LDB division algebra with dimension $n \geq 3$. Then, $\calT_A$ is reflexive.
\end{prop}

\begin{proof}
Denote by $q$ the quadratic form attached to $(A,\star,\bullet)$.
Let $f \in \calR(\calT_A)$. Since our goal is to prove that $f \in \calT_A$,
we will subtract well-chosen elements of $\calT_A$ to $f$ so as to obtain $0$.
First of all, we find four endomorphisms $g$, $h$, $i$ and $j$ of $A$ such that
$$\forall (y,z) \in A^2, \; f(y,z)=\bigl(g(y)+i(z)\,,\,j(y)+h(z)\bigr).$$
Let $y \in A$. Then, $f(y,0)=(g(y),j(y))$ belongs to
$\bigl\{(\lambda y,x \bullet y) \mid (x,\lambda)\in A \times \K\bigr\}$, and hence $g(y)$ is a scalar multiple of $y$.
Thus, $g=\alpha\id_A$ for some $\alpha \in \K$, and the same line of reasoning applied to the pairs $(0,z)$
shows that $h=\beta\id_A$ for some $\beta \in \K$. Subtracting from $f$ the operator $(y,z) \mapsto (\alpha y,\beta z)$,
which belongs to $\calT_A$, we see that no generality is lost in assuming that
$$\forall (y,z) \in A^2, \; f(y,z)=\bigl(i(z)\,,\,j(y)\bigr).$$

Let us choose an arbitrary non-zero vector $z_0 \in A$, and let $x_0 \in A$ be such that
$x_0 \star z_0=i(z_0)$. Then, subtracting the operator $(y,z) \mapsto (x_0 \star z, x_0 \bullet y)$ from $f$, we see that
no further generality is lost in assuming that $i(z_0)=0$.
Let us examine the image of $j$. Let $y \in A$. Then, there must be a triple $(x,\lambda,\mu) \in A \times \K^2$
such that
$$(0,j(y))=f(y,z_0)=\bigl(x \star z_0+\lambda y\, , \, x \bullet y+\mu z_0\bigr).$$
If $\lambda =0$, then one finds $x \star z_0=0$ with $z_0 \neq 0$, and hence $x=0$ and $j(y)=\mu z_0$. \\
Assume now that $j(y) \not\in \K z_0$. Then, $\lambda \neq 0$ and hence
$$x \bullet y=-\frac{1}{\lambda}\, x \bullet (x \star z_0)=-\frac{q(x)}{\lambda}\,z_0,$$
and hence $j(y) \in \K z_0$. In any case, one finds that $\im j \subset \K z_0$.
As $\star$ is a quasi-left-inversion of $\bullet$, this method proves in general that
\begin{equation}\label{inclusionnoyauimage}
\forall z \in \Ker i \setminus \{0\}, \; \im j \subset \K z \quad \text{and} \quad
\forall y \in \Ker j \setminus \{0\}, \; \im i \subset \K y.
\end{equation}
One successively deduces that $\rk j \leq 1$ (using the first statement in \eqref{inclusionnoyauimage}
together with the assumption that $\Ker i \neq \{0\}$), and that $\dim \Ker j \geq 2$ (as $n \geq 3$);
using the second statement in \eqref{inclusionnoyauimage}, this leads to $i=0$, and then $j=0$
by using once more the first statement in \eqref{inclusionnoyauimage} (as $n \geq 2$).
Finally, $f=0$, which finishes the proof.
\end{proof}

\begin{cor}\label{reflexiveclosureofahyperplane}
Let $\calH$ be a non-isotropic hyperplane of the twisted operator space $\calT_A$ attached to an LDB division algebra
$(A,\star,\bullet)$ whose dimension is greater than $2$.
Then, $\calR(\calH)=\calT_A$.
\end{cor}

In the case $n=2$, the operator space attached to an LDB division algebra is never reflexive:

\begin{prop}
Let $(A,\star,\bullet)$ be a $2$-dimensional LDB division algebra. Then, in well-chosen bases of $A^2$,
the space $\calT_A$ is represented by a $4$-dimensional subspace of $\Mata_4(\K)$.
In such bases, $\calR(\calT_A)$ is represented by $\Mata_4(\K)$, which has dimension $6$.
\end{prop}

\begin{proof}
As we can replace $\calT_A$ with an equivalent subspace, we may use Proposition \ref{unicitedelinversion}
and the classification of $2$-dimensional division algebras to reduce the situation to two canonical cases:
\begin{itemize}
\item Case 1: $A$ is a separable quadratic extension of $\K$.
Denote by $\sigma$ the non-identity automorphism of the $\K$-algebra $A$.
The pairing $(x,y) \mapsto \sigma(x)y$ is obviously a bilinear quasi-left-inversion of the product, so that the attached operator space
$\calT_A$ is the set of all linear maps
$$(y,z)\in A^2 \mapsto \bigl(xz+\lambda y\, , \, \sigma(x)y+\mu z\bigr), \quad \text{with $(x,\lambda,\mu) \in A \times \K^2$.}$$
The (alternating) bilinear form on the $\K$-vector space $A^2$ defined by
$$B\bigl( (a,b)\,,\, (c,d)\bigr):=\sigma(a)c-a\sigma(c)+\sigma(b)d-b\sigma(d)$$
is obviously non-degenerate, and one checks that
$$\forall (y,z) \in A^2, \;\forall f \in \calT_A, \; B\bigl((y,z),f(y,z)\bigr)=0.$$
Thus, in bases $\calB$ and $\calC$ of $A^2$ in which the bilinear form $B$ is represented by $I_4$,
the space $\calT_A$ is represented by a $4$-dimensional subspace of $\Mata_4(\K)$.

\item Case 2: $A$ is an inseparable quadratic extension of $\K$.
Then, the product on $A$ is a bilinear quasi-left-inversion of itself, and the attached operator space $\calT_A$
is the set of all linear maps
$$(y,z)\in A^2 \mapsto \bigl(xz+\lambda y\, , \, xy+\mu z\bigr), \quad \text{with $(x,\lambda,\mu) \in A \times \K^2$.}$$
Let us choose a linear form $\alpha : A \rightarrow \K$ with kernel $\K$.
One checks that
$$B : ((a,b),(c,d)) \mapsto \alpha(ac+bd)$$
is a non-degenerate bilinear form on the $\K$-vector space $A^2$.
Using the fact that $y^2 \in \K$ for all $y \in A$, one sees that
$$\forall (y,z)\in A^2, \; \forall f \in \calT_A, \; B\bigl((y,z),f(y,z)\bigr)=0.$$
As in Case 1, this shows that $\calT_A$ is represented by a $4$-dimensional subspace of $\Mata_4(\K)$ in well-chosen bases of $A^2$.
\end{itemize}
Remember from Proposition \ref{LDBdonneLLD} that $\dim \calT_A Y=3$ for all $Y \in A^2 \setminus \{0\}$.
In order to conclude, we let $\calV$ be an arbitrary $4$-dimensional linear subspace of $\Mata_4(\K)$ -
seen as a linear subspace of $\calL(\K^4)$ - satisfying $\dim \calV X=3$ for all non-zero vectors $X \in \K^4$.
Then, for the standard non-degenerate symmetric bilinear form on $\K^4$, we have $\calV X=\{X\}^\bot$ for all non-zero vectors $X \in \K^4$, and hence
the reflexive closure of $\calV$ is the set of all matrices $A \in \Mat_4(\K)$ satisfying $X^TAX=0$ for all non-zero
vectors $X \in \K^4$,
i.e. $\calR(\calV)=\Mata_4(\K)$.
\end{proof}

In order to find a non-isotropic hyperplane of $\calT_A$, it \emph{suffices}
to find a non-zero value $a \in \K$ outside the range of $q$.
Then, the quadratic form $q \bot \langle -a\rangle$ is non-isotropic, and hence
the hyperplane of all operators $\Gamma_A(x+\lambda(a,1),-)$ with $(x,\lambda)\in A \times \K$
is non-isotropic. For fields of characteristic not $2$, we shall prove that a
non-isotropic hyperplane of $\calT_A$ is always equivalent to a hyperplane of that type.

Now, we may at last give explicit examples of fields for which the upper bound $2n-2$ in
Theorem \ref{minranktheo1} is optimal, for some values of $n$ that are greater than $2$.

Assume that there are two non-zero scalars $a$ and $b$ in $\K$ such that $\langle 1,-a\rangle \otimes \langle 1,-b\rangle$
is non-universal (i.e.\ its range does not contain every non-zero scalar)
and non-isotropic (note that if $\K$ has characteristic not $2$, a non-degenerate isotropic form is always universal).
Then, $\langle 1,-a\rangle \otimes \langle 1,-b\rangle$
is equivalent to the norm $N$ of the quaternion algebra $A=C(\langle a,b\rangle)$ over $\K$.
Choosing a value $c$ outside the range of $\langle 1,-a\rangle \otimes \langle 1,-b\rangle$,
one sees that the quadratic form $\widetilde{N}$ is non-isotropic on the subspace of $\calT_A$ corresponding to
$C(\langle a,b\rangle) \times \K (1,c)$; this yields a $5$-dimensional non-reflexive subspace of $\calT_A$ in which
all the non-zero operators have rank $8$.
Using $c$, we may also construct an $8$-dimensional LDB division algebra $B$
whose attached quadratic form is equivalent to $\langle 1,-a\rangle \otimes \langle 1,-b\rangle \otimes \langle 1,-c\rangle$
(see the construction of octonions recalled in Section \ref{LDBdef}), and, if we can find a value $d$ outside
the range of this form, then we can find a $9$-dimensional non-reflexive operator space in which
all the non-zero operators have rank $16$.

For the field of reals, the above constructions with $a=b=c=d=-1$ give rise to examples for $n=5$ and $n=9$,
respectively associated with standard quaternions and standard octonions.
Note that this kind of construction may also be done for non-real fields: for example, one may take
the quotient field $\C((a,b,c,d))$ of the ring of formal power series in four independent variables $a,b,c,d$ and complex coefficients; over this field,
the quadratic form $\langle 1,-a\rangle \otimes \langle 1,-b\rangle \otimes \langle 1,-c\rangle \otimes \langle 1,-d\rangle$
is non-isotropic.

For some fields of characteristic $2$, it is possible to construct LDB division algebras of arbitrary large dimension.
Consider a list $(t_1,\dots,t_n,t_{n+1})$ of independent indeterminates, and denote
by $\L=\F_2(t_1,\dots,t_{n+1})$
the field of fractions of the corresponding polynomial ring (with coefficients in $\F_2$). Then,
$\K:=\F_2(t_1^2,\dots,t_n^2,t_{n+1})$ is a subfield of $\L$, and $\L$ has dimension $2^n$ over $\K$.
Considering $\L$ as a $\K$-algebra, we note that the two pairings $\star : (x,y) \mapsto xy$ and $\bullet:=\star$
are regular and $\forall (x,y)\in \L^2, \; x \star (x\bullet y)=x^2 y$, and $x \mapsto x^2$
is a non-isotropic quadratic form on the $\K$-vector space $\L$. Thus, $(\L,\star,\bullet)$
is an LDB division algebra over $\K$ with attached quadratic form $q: x \mapsto x^2$.
Note that this quadratic form is totally degenerate, i.e.\ it is additive!
Note also that $q$ is equivalent to the quadratic form $\langle 1,t_1^2\rangle \otimes \cdots \otimes \langle 1,t_n^2\rangle$.
Obviously, the scalar $t_{n+1}$ of $\K$ does not belong to the range of $q$, and hence
$\calT_\L$ contains a non-isotropic hyperplane.

For fields of characteristic not $2$, the existence of LDB division algebras with dimension larger than $8$ remains an open issue.

\subsection{On the equivalence between twisted operator spaces associated with LDB division algebras}\label{caracequivalencesection}

Here, we delve deeper into the structure of the twisted operator space attached to
an LDB division algebra.

First of all, we introduce two notions of similarity between LDB division algebras.
Let $(A,\star,\bullet)$ be an LDB division algebra with attached quadratic form $q$, and $B$ be a vector space such that $\dim B=\dim A$.
Let $f : B \overset{\simeq}{\rightarrow} A$, $g : B \overset{\simeq}{\rightarrow} A$
and $h : A \overset{\simeq}{\rightarrow} B$ be isomorphisms. The
bilinear mapping $\star'$ defined on $B^2$ by
$$x \star' y:=h(f(x) \star g(y))$$
is obviously equivalent to $\star$, and hence regular, and one checks that the bilinear mapping $\bullet'$ defined on $B^2$ by
$$x \bullet' y:=g^{-1}(f(x) \bullet h^{-1}(y))$$
is a quasi-left-inversion of $\star'$ and that $x \mapsto q(f(x))$ is the quadratic form attached to the LDB division algebra
$(B,\star',\bullet')$. This motivates the following definitions:

\begin{Def}
Let $(A,\star,\bullet)$ and $(B,\star',\bullet')$ be LDB division algebras.
\begin{itemize}
\item We say that $(A,\star,\bullet)$ and $(B,\star',\bullet')$ are \textbf{weakly equivalent}
when the division algebras $(A,\star)$ and $(B,\star')$ are equivalent, i.e.\ when
there are isomorphisms
$f : B \overset{\simeq}{\rightarrow} A$, $g : B \overset{\simeq}{\rightarrow} A$ and $h : A \overset{\simeq}{\rightarrow} B$ such that
$$\forall (x,y)\in B^2, \; x \star' y=h(f(x) \star g(y)).$$
\item We say that $(A,\star,\bullet)$ and $(B,\star',\bullet')$ are \textbf{equivalent}
when there are isomorphisms
$f : B \overset{\simeq}{\rightarrow} A$, $g : B \overset{\simeq}{\rightarrow} A$ and $h : A \overset{\simeq}{\rightarrow} B$ such that
$$\forall (x,y)\in B^2, \; x \star' y=h(f(x) \star g(y)).$$
and
$$\forall (x,y)\in B^2, \; x \bullet' y=g^{-1}(f(x) \bullet h^{-1}(y)).$$
\end{itemize}
\end{Def}

Obviously, the equivalence between two LDB division algebras implies their weak equivalence.
Using Proposition \ref{unicitedelinversion}, one sees that if $(A,\star,\bullet)$ is weakly equivalent
to $(B,\star',\bullet')$, then, for some $\alpha \in \K \setminus \{0\}$,
the LDB division algebras $(A,\star,\bullet)$ and $(B,\star',\alpha\,\bullet')$ are equivalent.

Our aim in the rest of the section is to study the relationship between those properties and the
structure of the associated twisted operator spaces. Here are two interesting issues:
\begin{itemize}
\item On what conditions on $A$ and $B$ are the operator spaces $\calT_A$ and $\calT_B$ equivalent?
\item On what conditions on $A$ and $B$ are the endomorphism spaces $\calT_A$ and $\calT_B$ similar?
\end{itemize}

The following theorem answers both questions:

\begin{theo}\label{classtwistedtheo}
Let $A$ and $B$ be LDB division algebras.
Then:
\begin{enumerate}[(a)]
\item The LDB division algebras $A$ and $B$ are weakly equivalent if and only if
the operator spaces $\calT_A$ and $\calT_B$ are equivalent.
\item The LDB division algebras $A$ and $B$ are equivalent if and only if
the endomorphism spaces $\calT_A$ and $\calT_B$ are similar.
\end{enumerate}
\end{theo}

Proving the converse implications in Theorem \ref{classtwistedtheo} is rather straightforward:
assume indeed that $(A,\star,\bullet)$ and $(B,\star',\bullet')$ are equivalent.
Then, we have three isomorphisms
$f : B \overset{\simeq}{\rightarrow} A$, $g : B \overset{\simeq}{\rightarrow} A$ and $h : A \overset{\simeq}{\rightarrow} B$
such that
$$\forall (x,y)\in B^2, \; x \star' y=h\bigl(f(x) \star g(y)\bigr).$$
and
$$\forall (x,y)\in B^2, \; x \bullet' y=g^{-1}\bigl(f(x) \bullet h^{-1}(y)\bigr).$$
Thus, for every $(x,y,z,\lambda,\mu)\in B^3 \times \K^2$, one finds that
$$\Gamma_B(x+(\lambda,\mu),(y,z))
=\Bigl(h\bigl(f(x)\star g(z)+\lambda\, h^{-1}(y)\bigr), g^{-1}\bigl(f(x) \bullet h^{-1}(y) +\mu\, g(z)\bigr)\Bigr).$$
Therefore, the two isomorphisms
$$F : x+(\lambda,\mu) \in B \oplus \K^2 \longmapsto f(x)+(\lambda,\mu) \in A \oplus \K^2$$
and
$$G : (y,z)\in B^2 \longmapsto (h^{-1}(y),g(z)) \in A^2$$
satisfy
$$\forall X \in B \oplus \K^2, \quad
\Gamma_B(X,-)=G^{-1} \circ \Gamma_A(F(X),-) \circ G,$$
and hence $\calT_A$ and $\calT_B$ are similar.

With the weaker assumption that $A$ is weakly equivalent to $B$,
we would find some $\alpha \in \K \setminus \{0\}$ such that $(B,\star',\bullet')$ is equivalent
to $(A,\star,\alpha\bullet)$; as $\calT_{(A,\star,\alpha\,\bullet)} \sim \calT_{A,\star,\bullet}$,
one would deduce that $\calT_{B,\star',\bullet'} \sim \calT_{A,\star,\bullet}$.

\vskip 4mm
In contrast, the proof of the direct implications in Theorem \ref{classtwistedtheo} is highly
non-trivial. It is based upon the following technical result, of which we will later derive other consequences.

\begin{lemme}[Rectification lemma]\label{rectificationlemma}
Let $(A,\star,\bullet)$ be an LDB division algebra with attached quadratic form $q$.
One endows $A \oplus \K^2$ with the quadratic form $\widetilde{q} : x+(\lambda,\mu) \mapsto q(x)-\lambda \mu$.
Let $s$ be an orthogonal automorphism of the quadratic space $(A \oplus \K^2,\widetilde{q})$.
Then, there are automorphisms $F : A^2 \overset{\simeq}{\rightarrow} A^2$,
$G : A^2 \overset{\simeq}{\rightarrow} A^2$ and an orthogonal automorphism $H$ of $A \oplus \K^2$
such that $H$ fixes every vector of $\K^2$ and
$$\forall x \in A \oplus \K^2, \quad \Gamma_A(s(x),-)=G \circ \Gamma_A(H(x),-) \circ F.$$
\end{lemme}

Proving this lemma requires the following result on the orthogonal group of~$\widetilde{q}$.
Recall that, given a quadratic space $(E,\varphi)$ with polar form $b : (x,y) \mapsto \varphi(x+y)-\varphi(x)-\varphi(y)$,
together with a non-isotropic vector $a \in E$, the reflection along $\K a$ is defined as
the linear map $x \mapsto x-\frac{b(x,a)}{\varphi(a)}\,a$.

\begin{lemme}\label{productofreflexions}
The orthogonal group of the quadratic space $(A \oplus \K^2, \widetilde{q})$
is generated by the reflections along the non-isotropic lines of $A \oplus \K^2$
that are not included in $\K^2$.
\end{lemme}

\begin{proof}[Proof of Lemma \ref{productofreflexions}]
First of all, we show that the orthogonal group of $(A \oplus \K^2, \widetilde{q})$
is generated by reflections.

As $\widetilde{q}$ is regular on $\K^2$, and as $\K^2$ is $\widetilde{q}$-orthogonal
to $A$, the radical of $\widetilde{q}$ is included in $A$, and hence $\widetilde{q}$ is non-isotropic on its radical.
It is a general fact that whenever a finite-dimensional quadratic form $\varphi$ is non-isotropic on its radical,
its orthogonal group is generated by reflections, unless $\varphi$ is a $4$-dimensional hyperbolic form over
a field with two elements. If the underlying field has characteristic not $2$, then
$\varphi$ is regular as it vanishes everywhere on its radical, and hence the result is simply Witt's theorem,
see \cite[Chapter XXI, Theorem 2.1.1]{invitquad}.
When $\K$ has characteristic $2$, there are two subcases:
\begin{itemize}
\item \textbf{Case 1:} $\K$ has more than $2$ elements. Then, one can use
an equivalent of Witt's extension lemma (see \cite[Lemma 5.2.3]{invitquad}) to reduce the situation to the one where
$\varphi$ is totally isotropic, i.e.\ its radical is the source space $E$ of $\varphi$.
In that special case however, $\varphi$ is a linear injection from $E$ to $\K$ (where the scalar multiplication on $\K$ is
defined by $\lambda.x:=\lambda^2 x$)
and hence the only orthogonal automorphism for $\varphi$ is the identity.
\item \textbf{Case 2:} $\K \simeq \F_2$. Then, the radical of $\varphi$ has dimension $0$ or $1$
because $\varphi$ is linear on it, and hence one-to-one as it is non-isotropic on it.
In the first case, the result is a theorem of Dieudonn\'e \cite[Proposition 14, p.42]{Dieudonnegroupesclassiques}.
In the second one, it is known that the orthogonal group of $\varphi$ is isomorphic to a symplectic group,
through an isomorphism that turns reflections into symplectic transvections \cite[Section V.23]{Dieudonnegroupesclassiques},
and one concludes by using the classical theorem stating that every symplectic group is generated by symplectic transvections
\cite[Proposition 4, p.10]{Dieudonnegroupesclassiques}.
\end{itemize}

Now, for a non-isotropic vector $x$ of $(A \oplus \K^2, \widetilde{q})$,
denote by $s_x$ the reflection along $\K x$.
In order to conclude, we fix an arbitrary non-isotropic vector $a \in \K^2$
and we prove that the reflection $s_a$ is a product of reflections along vectors of $(A \oplus \K^2) \setminus \K^2$.
However, given a non-isotropic vector $b$ of $A \oplus \K^2$, one computes that
$s_a=s_b^{-1} \circ s_{s_b(a)} \circ s_b$. If $b$ is not orthogonal to $a$, then $b \in \Vect(a,s_b(a))$.
In order to conclude, it suffices to show that $b$ may be chosen outside of $\K^2$ and non-orthogonal to $a$,
so that $s_b(a)$ cannot belong to $\K^2$. Note that $b_0=(1,0)$ is not orthogonal to $a$ since $a$ is non-isotropic;
choosing an arbitrary vector $x_0 \in A \setminus \{0\}$, we see that
$b:=x_0+b_0$ is not orthogonal to $a$ and $\widetilde{q}(b)=q(x_0) \neq 0$,
which completes the proof since $b\not\in \K^2$.
\end{proof}

\begin{proof}[Proof of Lemma \ref{rectificationlemma}]
We start with the case when $s$ is the reflection along a non-isotropic vector
$X_0=x_0+(\lambda_0,\mu_0)$ of $(A \oplus \K^2) \setminus \K^2$.
Then, we have $x_0 \neq 0$ and hence $q(x_0) \neq 0$.
Set
$$\overrightarrow{a}:=s\bigl((1,0)\bigr) \quad \text{and} \quad
\overrightarrow{b}:=s\bigl((0,1)\bigr),$$
$$f:=\Gamma_A(\overrightarrow{a},-) \quad \text{and} \quad g:=\Gamma_A(\overrightarrow{b},-).$$
Setting $b_q : (y,z) \in A^2 \mapsto q(y+z)-q(y)-q(z)$, we have
$$\forall (x,\lambda,\mu) \in A \times \K^2, \; s(x+(\lambda,\mu))
=x+(\lambda,\mu)-\frac{b_q(x,x_0)-\lambda\mu_0-\mu \lambda_0}{q(x_0)-\lambda_0\mu_0}\bigl(x_0+(\lambda_0,\mu_0)\bigr).$$
In particular,
$$\overrightarrow{a}=\frac{\mu_0}{q(x_0)-\lambda_0\mu_0}\,x_0+\Bigl(\frac{q(x_0)}{q(x_0)-\lambda_0\mu_0},\frac{\mu_0^2}{q(x_0)-\lambda_0\mu_0}\Bigr)$$
and
$$\overrightarrow{b}=\frac{\lambda_0}{q(x_0)-\lambda_0\mu_0}\,x_0+\Bigl(\frac{\lambda_0^2}{q(x_0)-\lambda_0\mu_0},\frac{q(x_0)}{q(x_0)-\lambda_0\mu_0}\Bigr).$$
One computes that
\begin{align*}
\Ker f& = & \bigl\{(y,z) \in A^2: \; q(x_0)\,y=-\mu_0\,x_0\star z\bigr\} = & \bigl\{(\mu_0\, x_0 \star t,-q(x_0)\,t) \mid t \in A\bigr\} \\
\Ker g& = & \bigl\{(y,z) \in A^2: \; q(x_0)\,z=-\lambda_0\,x_0\bullet y\bigr\} = & \bigl\{(-q(x_0)\,t,\lambda_0\, x_0 \bullet t) \mid t \in A\bigr\}.
\end{align*}
Indeed, if we consider $f$ for example, the identity
$\bigl\{(y,z) \in A^2: \; q(x_0)\,y=-\mu_0\,x_0\star z\bigr\} =  \bigl\{(\mu_0\, x_0 \star t,-q(x_0)\,t) \mid t \in A\bigr\}$
and the inclusion $\Ker f \subset \bigl\{(y,z) \in A^2: \; q(x_0)\,y=-\mu_0\,x_0\star z\bigr\}$
are obvious; as $s$ is a reflection and $(1,0)$ is $\widetilde{q}$-isotropic, $\overrightarrow{a}$ is also $\widetilde{q}$-isotropic;
then, Proposition \ref{rankinatwisted} yields $\rk f=\dim A$; one concludes by noting that $\bigl\{(\mu_0\, x_0 \star t,-q(x_0)\,t) \mid t \in A\bigr\}$
has dimension $\dim A$, visibly. The case of $g$ is handled similarly.

Noting that $\overrightarrow{a}+\overrightarrow{b}=s\bigl((1,1)\bigr)$ is non-isotropic,
we find that $f+g$ is non-singular, and hence $\im f+\im g=A^2$
and $\Ker f \cap \Ker g=\{0\}$. As $\rk f=\rk g=\frac{\dim A^2}{2}$, one deduces that $\Ker f\oplus \Ker g=A^2=\im f\oplus \im g$.

Now, let us consider the isomorphism $g_1 : A \overset{\simeq}{\rightarrow} \im f$ (respectively, $g_2 : A \overset{\simeq}{\rightarrow} \im g$)
obtained by composing the isomorphism
$$f_1 : y \in A \mapsto (q(x_0)\,y\, ,\,-\lambda_0\,x_0\bullet y) \in \Ker g$$
with the isomorphism $\Ker g \overset{\simeq}{\rightarrow} \im f$ induced by $f$
(respectively, by composing the isomorphism
$$f_2 : z \in A \mapsto (-\mu_0\,x_0\star z\, ,\, q(x_0)\,z) \in \Ker f$$
with the isomorphism $\Ker f \overset{\simeq}{\rightarrow} \im g$ induced by $g$).
One computes that
$$\forall y \in A, \; g_1(y)=\Bigl(q(x_0)\,y \,, \,  \mu_0\,x_0 \bullet y    \Bigr) \quad \text{and}
\quad
\forall z \in A, \; g_2(z)=\Bigl(\lambda_0\,x_0 \star z \, ,\,  q(x_0)\,z\Bigr).$$
Set
$$G : (y,z) \in A^2 \longmapsto g_1(y)+g_2(z) \in A^2 \quad \text{and} \quad
F : (y,z) \in A^2 \longmapsto f_1(y)+f_2(z) \in A^2.$$
As $A^2=\Ker f\oplus \Ker g=\im f\oplus \im g$, both maps $F$ and $G$ are automorphisms of $A^2$.

Now, let $x \in A$ and set $\alpha:=b_q(x_0,x)$.
Then,
$$s(x)=\Bigl(x-\frac{\alpha}{q(x_0)-\lambda_0\mu_0}\,x_0\Bigr)-\frac{\alpha}{q(x_0)-\lambda_0\mu_0}\,(\lambda_0,\mu_0).$$
Fixing $y \in A$, one
computes that
$$\Gamma_A(s(x),f_1(y))=\Bigl(-\lambda_0\,x\star (x_0 \bullet y)\, ,\,(q(x_0)\,x-b_q(x_0,x)\,x_0)\bullet y\Bigr).$$
Writing $(y_1,y_2)=\Gamma_A(s(x),f_1(y))$, one finds that
\begin{align*}
\frac{1}{q(x_0)}\,x_0 \star y_2 & =
\frac{1}{q(x_0)}\,\bigl(q(x_0)\,x_0 \star (x \bullet y)-b_q(x_0,x)\,x_0 \star (x_0 \bullet y)\bigr) \\
& = \bigl(x_0 \star (x \bullet y)-b_q(x_0,x)\, y\bigr) \\
& = -x \star (x_0 \bullet y) \\
& = \frac{1}{\lambda_0}\,y_1,
\end{align*}
where the identity $b_q(x_0,x)\, y =x \star (x_0 \bullet y)+x_0 \star (x \bullet y)$ comes from polarizing the
quadratic identity $\forall t \in A, \; t \star (t \bullet y)=q(t)\,y$.
With this, one sees that $\Gamma_A(s(x),f_1(y)) \in \im g_2$, and more precisely that
$$G^{-1}(\Gamma_A(s(x),f_1(y)))=\bigl(0,r(x) \bullet y\bigr),$$
where $r : t \mapsto t-\frac{b_q(t,x_0)}{q(x_0)}x_0$ denotes the reflection of $(A,q)$ along $\K x_0$.
Similarly, one computes that
$$\forall z \in A, \; G^{-1}(\Gamma_A(s(x),f_2(z)))=\bigl(r(x) \star z,0\bigr).$$
One concludes that
$$\Gamma_A(s(x),-)=G \circ \Gamma_A(r(x),-) \circ F^{-1}.$$
Finally, it is obvious from the definitions of $F$ and $G$ that
$$\Gamma_A\bigl(s\bigl((1,0)\bigr),-\bigr)=f=G \circ \Gamma_A\bigl((1,0),-\bigr) \circ F^{-1}$$
and
$$\Gamma_A\bigl(s\bigl((0,1)\bigr),-\bigr)=g=G \circ \Gamma_A\bigl((0,1),-\bigr) \circ F^{-1}.$$
Therefore, the isomorphisms $F^{-1}$, $G$ and the orthogonal automorphism
$H : x+(\lambda,\mu) \in A \oplus \K^2  \mapsto r(x)+(\lambda,\mu) \in A \oplus \K^2$ satisfy the claimed properties.

\vskip 3mm
Let us return to the general case. Let $u$ be an orthogonal automorphism of $A \oplus \K^2$
and $s$ be a reflection of $A\oplus \K^2$ along a line which is not included in $\K^2$.
Assume that there are two automorphisms $F$ and $G$ of $A^2$ together with
an orthogonal automorphism $H$ of $A \oplus \K^2$ which fixes every vector of $\K^2$ and satisfies
$$\forall x \in A \oplus \K^2, \quad \Gamma_A(u(x),-)=G \circ \Gamma_A(H(x),-) \circ F.$$
Then, we prove that the claimed result also holds for $u \circ s$.

Let us endow $B:=A$ with a new structure of LDB division algebra:
for all $(x,y) \in B^2$, one sets:
$$x \star' y:=H(x) \star y \quad \text{and} \quad x \bullet' y:=H(x) \bullet y.$$
Thus, we have
$$\forall x \in A \oplus \K^2, \; \Gamma_B(x,-)=\Gamma_A(H(x),-),$$
whence
$$\forall x \in A \oplus \K^2, \; \Gamma_A(u(x),-)=G \circ \Gamma_B(x,-) \circ F$$
and $q$ is the quadratic form attached to $B$!
Thus, $s$ is a reflection of the quadratic space $(B \oplus \K^2,\widetilde{q})$ along a line which is not included in $\K^2$.
Applying the first step in the LDB division algebra $(B,\star',\bullet')$
yields two automorphisms $F'$ and $G'$ of $A^2$ together with an orthogonal automorphism $H'$ of $(A \oplus \K^2,\widetilde{q})$
such that $H'$ fixes every vector of $\K^2$ and
$$\forall x \in B \oplus \K^2, \; \Gamma_B(s(x),-)=G' \circ \Gamma_B (H'(x),-) \circ F'.$$
For all $x \in A \oplus \K^2$, this yields
\begin{align*}
\Gamma_A(u(s(x)),-) & = G \circ \Gamma_B(s(x),-) \circ F \\
& = G \circ G' \circ \Gamma_B(H'(x),-) \circ F' \circ F \\
& = (G \circ G') \circ \Gamma_A \bigl((H\circ H')(x),-\bigr) \circ (F' \circ F).
\end{align*}
Since $H \circ H'$ is an orthogonal automorphism of $(A \oplus \K^2,\widetilde{q})$ that fixes every vector of $\K^2$,
we find that $u \circ s$ has the claimed property.

Finally, Lemma \ref{productofreflexions} shows that every orthogonal automorphism of $(A \oplus \K^2,\widetilde{q})$
is the product of reflections along lines that are not included in $\K^2$, and hence the result follows by induction
on the number of such reflections.
\end{proof}

\begin{proof}[Proof of the direct implications in Theorem \ref{classtwistedtheo}]
We have already proved the converse statements.
Let us write the considered LDB division algebras formally as $(A,\star,\bullet)$ and $(B,\star',\bullet')$, and
denote the attached quadratic forms by $q_A$ and $q_B$, respectively.

Assume that $\calT_A$ is equivalent to $\calT_B$.
Then, we have isomorphisms $F : A^2 \overset{\simeq}{\rightarrow} B^2$ and $G : B^2 \overset{\simeq}{\rightarrow} A^2$,
and an isomorphism $K : A \oplus \K^2 \overset{\simeq}{\rightarrow} B \oplus \K^2$ such that
$$\forall x \in A \oplus \K^2, \; \Gamma_A(x,-)=G \circ \Gamma_B(K(x),-)\circ F.$$
Let us define $\varphi : \calT_A \rightarrow \calT_B$ as the sole isomorphism which satisfies
$$\forall x \in A \oplus \K^2, \; \varphi(\Gamma_A(x,-))=\Gamma_B(K(x),-).$$

Then, $\varphi : \calT_A \rightarrow \calT_B$ is a rank-preserving isomorphism of vector spaces, and hence
the linear map $K : A \oplus \K^2 \rightarrow B \oplus \K^2$ maps the set of isotropic vectors of
$\widetilde{q_A}$ bijectively onto the set of isotropic vectors of $\widetilde{q_B}$.
As each form $\widetilde{q_A}$ and $\widetilde{q_B}$ has a hyperbolic sub-form,
the quadratic Nullstellensatz yields a non-zero scalar
$\alpha$ such that $\widetilde{q_B}(K(x))=\alpha\,\widetilde{q_A}(x)$ for all $x \in A \oplus \K^2$.
Set $x:=K\bigl((1,0)\bigr)$ and $y:=K\bigl((0,-\alpha^{-1})\bigr)$.
Then, $x$ and $y$ are $\widetilde{q_B}$-isotropic and
$\widetilde{q_B}(x+y)=\alpha\, \widetilde{q_A}((1,-\alpha^{-1}))=1$, so that
$(x,y)$ is a hyperbolic pair of the quadratic space $(B \oplus \K^2,\widetilde{q_B})$.
By Witt's extension theorem, there is an orthogonal automorphism $u$ of $(B \oplus \K^2,\widetilde{q_B})$
such that $K\bigl((1,0)\bigr)=u\bigl((1,0)\bigr)$ and $K\bigl((0,-\alpha^{-1})\bigr)=u\bigl((0,-1)\bigr)$.
Then, by Lemma \ref{rectificationlemma}, we may find automorphisms
$F' : B^2 \rightarrow B^2$ and $G' : B^2 \rightarrow B^2$ together with an orthogonal automorphism $H$ of $B \oplus \K^2$ such that
$H$ fixes every vector of $\K^2$ and
$$\forall x \in B \oplus \K^2, \; \Gamma_B(u(x),-)=G' \circ \Gamma_B(H(x),-) \circ F'.$$
This leads to
$$\forall x \in A \oplus \K^2, \; \Gamma_A(x,-)=(G \circ G') \circ \Gamma_B(H(u^{-1}(K(x))),-) \circ (F' \circ F),$$
and we note that $G_1:=G \circ G'$ and $F_1:=F' \circ F$ are isomorphisms from $B^2$ to $A^2$ and from $A^2$ to $B^2$, respectively.
The definition of $u$ shows that $(u^{-1}\circ K)\bigl((1,0)\bigr)=(1,0)$ and $(u^{-1}\circ K)\bigl((0,1)\bigr)=(0,\alpha)$.
As $u$ and $K$ preserve orthogonality, this yields an isomorphism $k : A \overset{\simeq}{\rightarrow} B$
such that $u^{-1}(K(x))=k(x)$ for all $x \in A$. Setting finally $h :=H \circ k$, one finds that $h : A \rightarrow B$
is an isomorphism and
$$\forall (x,\lambda,\mu) \in A \times \K^2, \quad \Gamma_A\bigl(x+(\lambda,\mu),-\bigr)=
G_1 \circ \Gamma_B\bigl(h(x)+(\lambda,\alpha\mu),-\bigr) \circ F_1.$$
As $\im \Gamma_A\bigl((1,0),-\bigr)=A \times \{0\}$ and $\im \Gamma_B\bigl((1,0),-\bigr)=B \times \{0\}$, one deduces that
$G_1$ maps $B \times \{0\}$ into $A \times \{0\}$. Similarly, one proves that $G_1$ also maps $\{0\} \times B$
into $\{0\} \times A$.
Looking at the respective kernels of $\Gamma_A\bigl((0,1),-\bigr)$, $\Gamma_B\bigl((0,\alpha),-\bigr)$, $\Gamma_A\bigl((1,0),-\bigr)$ and
$\Gamma_B\bigl((1,0),-\bigr)$,
one shows in a similar fashion that $F_1$ maps $A \times \{0\}$ into $B \times \{0\}$ and $\{0\} \times A$ into $\{0\} \times B$.
Thus, we have four isomorphisms $f_1 : A \overset{\simeq}{\rightarrow} B$, $f_2 : A \overset{\simeq}{\rightarrow} B$,
$g_1 : B \overset{\simeq}{\rightarrow} A$ and $g_2 : B \overset{\simeq}{\rightarrow} A$ such that
$$\forall (y,z) \in A^2, \; F_1(y,z)=(f_1(y),f_2(z)) \quad \text{and} \quad
\forall (y,z) \in B^2, \; G_1(y,z)=(g_1(y),g_2(z)).$$
Then, $G_1 \circ \Gamma_B\bigl((1,0),-\bigr) \circ F_1=\Gamma_A\bigl((1,0),-\bigr)$ and
$G_1 \circ \Gamma_B\bigl((0,\alpha),-\bigr) \circ F_1=\Gamma_A\bigl((0,1),-\bigr)$ yield
$g_1 \circ f_1=\id_A$ and $g_2 \circ f_2=\alpha^{-1} \id_A$, respectively.

Thus, for all $x \in A$, identity $\Gamma_A(x,-)=G_1 \circ \Gamma_B\bigl(h(x),-\bigr) \circ F_1$ means that
$$\forall (y,z) \in A^2, \; x \star z=f_1^{-1}\bigl(h(x) \star' f_2(z)\bigr) \quad \text{and}
\quad x \bullet y=\alpha^{-1}\,f_2^{-1}\bigl(h(x) \bullet' f_1(y)\bigr),$$
which shows that $A$ and $B$ are weakly equivalent LDB division algebras.

Moreover, if $\calT_A$ and $\calT_B$ are similar, the isomorphisms $F$ and $G$ from the beginning of the proof may be chosen so as to have
$G=F^{-1}$. In that case, as $\Gamma_A\bigl((1,1),-\bigr)=\id_{A^2}$, we find that $\Gamma_B\bigl(K((1,1)),-\bigr)=\id_{B^2}$, and
hence $K\bigl((1,1)\bigr)=(1,1)$.
Thus, we actually have $\alpha=1$, and the rest of the proof yields that $A$ and $B$ are equivalent LDB division algebras.
\end{proof}

Here is a stunning corollary:

\begin{cor}
Let $(A,\star,\bullet)$ be an LDB division algebra. Then, $(A,\bullet,\star)$ is equivalent to $(A,\star,\bullet)$.
\end{cor}

\begin{proof}
By Theorem \ref{classtwistedtheo}, it suffices to show that $\calT_{A,\bullet,\star}$
and $\calT_{A,\star,\bullet}$ are similar, which follows from the simple observation that
the ``twist" automorphism $T : (x,y) \in A^2 \mapsto (y,x) \in A^2$ satisfies
$$\forall (x,\lambda,\mu) \in A \times \K^2, \; \Gamma_{(A,\bullet,\star)}\bigl(x+(\lambda,\mu),-\bigr)
=T^{-1} \circ \Gamma_{(A,\star,\bullet)} \bigl(x+(\mu,\lambda),-\bigr) \circ T.$$
\end{proof}

We finish with another application of the Rectification lemma:

\begin{prop}\label{hyperplanereduced}
Let $(A,\star,\bullet)$ be an LDB division algebra, with attached quadratic form $q$, and assume that $\K$ has characteristic not $2$.
Let $\calH$ be a non-isotropic hyperplane of $\calT_A$. Then, there is a non-zero scalar $\alpha$
that is not in the range of $q$ and such that
$\calH$ is equivalent to the space of all operators of the form
$$(y,z) \in A^2 \mapsto \bigl(x \star z+\alpha \lambda y\, , \, x \bullet y+\lambda z\bigr) \in A^2, \quad
\text{for $(x,\lambda)\in A \times \K$.}$$
\end{prop}

\begin{proof}
Denote by $\calH_0$ the hyperplane of $A \oplus \K^2$ corresponding to $\calH$ through
$X \mapsto \Gamma_A(X,-)$.
Then, the orthogonal subspace $\calD$ of $\calH_0$ in $(A \oplus \K^2,\widetilde{q})$ is $1$-dimensional and non-isotropic.
Choosing a non-zero value $\alpha$ of $\widetilde{q}$ on $\calD$, we set $\calD':=\K (-\alpha,1) \subset A \oplus \K^2$
and we note that $\widetilde{q}$ takes the value $\alpha$ on $\calD'$.

Thus, the quadratic spaces $(\calD',q_{\calD'})$ and $(\calD,q_{\calD})$ are isomorphic, and hence
Witt's extension theorem yields an orthogonal automorphism $u$ of $(A \oplus \K^2,\widetilde{q})$ such that $u(\calD')=\calD$.
Using Lemma \ref{rectificationlemma}, one finds automorphisms $G$ and $F$ of $A^2$, together with
an orthogonal automorphism $v$ of $(A \oplus \K^2,\widetilde{q})$ that fixes every vector of $\K^2$ and for which
$$\forall x \in A \oplus \K^2, \; \Gamma_A(u(x),-)=G \circ \Gamma_A(v(x),-) \circ F.$$
Thus, $\calH$ is equivalent to the space of all operators $\Gamma_A(y,-)$ with
$y$ in the hyperplane $\calH_1:=(v\circ u^{-1})(\calH_0)$ of $A \oplus \K^2$.
As $v \circ u^{-1}$ is an orthogonal automorphism of $(A \oplus \K^2,\widetilde{q})$,
we see that $\calH_1$ is the orthogonal complement of $(v \circ u^{-1})(\calD)=v(\calD')=\calD'$, that is
$\calH_1=A\oplus \K (\alpha,1)$. Thus, $\calH$ is equivalent to the space $\calH'$ of all operators
$\Gamma_A\bigl(x+(\alpha \lambda,\lambda),-\bigr)$ with $(x,\lambda)\in A \times \K$.

Since all the non-zero operators in $\calH'$ must be non-singular, the quadratic form $\widetilde{q}$
is non-isotropic on $\calH_1$; as $\widetilde{q}_{|\calH_1} \simeq q \bot \langle -\alpha\rangle$, it ensues
that $\alpha$ does not belong to the range of $q$.
\end{proof}

\subsection{The operator space $\calT_A$ as a Jordan algebra}

We close this section with a short remark on the structure of the endomorphism space $\calT_A$.

Let $(A,\star,\bullet)$ be an LDB division algebra with attached quadratic form $q$. Viewing $\calT_A$ as a linear subspace of $\calL(A^2)$,
we note that it has the peculiar property of being stable by squares. Indeed, given $(x,\lambda,\mu) \in A \times \K^2$,
one checks that, for all $(y,z)\in A^2$,
$$\Gamma_A(x+(\lambda,\mu),\Gamma_A(x+(\lambda,\mu),(y,z)))^2=
\bigl((\lambda+\mu)x\star z+(\lambda^2+q(x))y\, , \, (\lambda+\mu)x\bullet y+(\mu^2+q(x))z\bigr),$$
and hence
$$\Gamma_A(x+(\lambda,\mu),-)^2=\Gamma_A\bigl((\lambda+\mu)x+(\lambda^2+q(x),\mu^2+q(x)),-\bigr).$$
In particular, $\calT_A$ is a Jordan subalgebra of $\calL(A^2)$ (for the Jordan product $(f,g) \mapsto fg+gf$),
and it is isomorphic to the Jordan algebra $(A\oplus \K^2,b)$, where
$$b\bigl(x+(\lambda,\mu),x'+(\lambda',\mu')\bigr)=
(\lambda'+\mu')x+(\lambda+\mu)x'+\bigl(2\lambda\lambda'+b_q(x,x'),2\mu\mu'+b_q(x,x')\bigr),$$
and $b_q :(x,x') \mapsto q(x+x')-q(x)-q(x')$ denotes the polar form of $q$.

\section{The minimal rank problem in a non-reflexive operator space (III)}\label{nonreflexivesectionIII}

In the preceding section, we have shown examples of $(n+1)$-dimensional LLD operator spaces containing
a hyperplane with minimal rank $2n-2$. In this final section, we show that those examples are essentially the only
possible spaces of this type, with the usual restrictions on the cardinality of the underlying field.

\subsection{The classification theorem}\label{classtheostatement}

\begin{theo}\label{theocaslimite}
Let $n \geq 3$ be an integer, and assume that $\# \K>n$.
Let $\calS$ be an $(n+1)$-dimensional reduced LLD subspace of $\calL(U,V)$, and
$\calH$ be a hyperplane of $\calS$ such that $\mrk \calH=2n-2$.
Then:
\begin{enumerate}[(a)]
\item $\calS$ is equivalent to the twisted operator space $\calT_A$ attached to some $(n-1)$-dimensional LDB division algebra $A$;
\item The LDB division algebra $A$ is uniquely defined by $\calS$ up to weak equivalence;
\item The operator space $\calH$ is equivalent to a non-isotropic hyperplane of $\calT_A$;
\item If $n \geq 4$, then $\calS$ is the reflexive closure of $\calH$ and
the weak equivalence class of the LDB division algebra $A$ is solely dependent on $\calH$.
\end{enumerate}
\end{theo}

As we have seen that LDB division algebras can only exist in even dimension or in dimension $1$,
it follows that Meshulam and \v Semrl's upper bound $2n-2$ from Theorem \ref{minranktheo1} can be replaced with $2n-3$ provided that
$n$ is even and larger than $3$.

A striking feature of the above classification theorem is that one only assumes that $\calS$ is LLD,
not that it lies within the reflexive closure of $\calH$!
The following corollary is straightforward by combining Theorem \ref{theocaslimite} with Proposition \ref{hyperplanereduced}.

\begin{cor}
Let $n \geq 3$ be an integer, and assume that $\# \K>n$ and that $\K$ has characteristic not $2$.
Let $\calS$ be an $(n+1)$-dimensional reduced LLD operator space, and
$\calH$ be a hyperplane of $\calS$ such that $\mrk \calH=2n-2$.
Then, there exists an LDB division algebra $A$ together with a non-zero scalar $\alpha$ such that
$\calH$ is equivalent to the space of all operators of the form
$$(y,z) \in A^2 \mapsto \bigl(x \star z+\alpha \,\lambda\, y\, , \, x \bullet y+\lambda\, z\bigr) \in A^2, \quad
\text{with $(x,\lambda)\in A \times \K$.}$$
Moreover, $\alpha$ is outside the range of the quadratic form attached to $A$.
\end{cor}

Finally, with the non-reflexivity assumption, we have found:

\begin{cor}\label{cornonreflexive}
Let $\calT$ be an $n$-dimensional reduced non-reflexive operator space, with $\# \K>n \geq 3$.
If all the non-zero operators in $\calT$ have rank at least $2n-2$, then
$\calT$ is equivalent to a hyperplane of the twisted operator space associated with an $(n-1)$-dimensional LDB division algebra.
\end{cor}

\subsection{Proof of the classification theorem}

Here, we prove Theorem \ref{theocaslimite}.
First of all, let us obtain points (b), (c) and (d) as easy consequences of point (a) together
with results from Section \ref{nonreflexivesectionII}.
To begin with, point (b) is a consequence of point (a) of Theorem \ref{classtwistedtheo}.

Now, assume that there is an LDB division algebra $A$, necessarily of dimension $n-1$, together with isomorphisms
$F : U \overset{\simeq}{\rightarrow} A^2$,
and $G : A^2 \overset{\simeq}{\rightarrow} V$ such that
$\calS=\bigl\{G \circ f \circ F \mid f \in \calT_A\bigr\}$.
As $f \mapsto G^{-1} \circ f \circ F^{-1}$ is linear, one-to-one and preserves the rank of operators, we see that
$\calH':=\bigl\{G^{-1} \circ f \circ F^{-1} \mid f \in \calH\bigr\}$ is a hyperplane of $\calT_A$ in which all the non-zero operators
have rank greater than or equal to $2n-2$. Thus, $\calH'$ is a non-isotropic hyperplane of $\calT_A$,
and it is obviously equivalent to $\calH$, proving point (c).

For any linear subspace $\calV$ of $\calL(U,V)$,
if we set $\calW:=\{G^{-1} \circ f \circ F^{-1} \mid f \in \calV\}$, then one sees that
$$\calR(\calV)=\bigl\{G \circ f \circ F \mid f \in \calR(\calW)\bigr\}.$$
In particular, if $n \geq 4$, then $\calT_A=\calR(\calH')$ by Corollary \ref{reflexiveclosureofahyperplane},
and hence $\calS=\calR(\calH)$; in that case, we see that the equivalence class of $\calS$ depends solely
on that of $\calH$, and point (b) yields that the weak equivalence class of $A$ depends solely on the equivalence class of $\calH$.
Thus, point (d) is proven.

It remains to prove point (a).
The basic strategy consists in pushing the arguments of our proof of Theorem \ref{minranktheo1} further.
If $\calH$ is LLD itself, then we know from Theorem \ref{basictheo} that $\mrk \calH<n$, contradicting $\mrk \calH=2n-2$.
Thus, $\calH$ is not LLD. Next, we see that $\calS$ is a minimal LLD space:
indeed, if it were not, then there would be an LLD hyperplane $\calH'$ of $\calS$; obviously we should have $\calH \neq \calH'$,
and hence $\calH \cap \calH'$ would be a hyperplane of $\calH'$; then
Theorem \ref{minranktheo1} would yield $\mrk(\calH) \leq \mrk (\calH \cap \calH') \leq 2n-4$.

Now, let us choose a rank-optimal operator $\varphi \in \widehat{\calS}$.
We see that $\Ker \varphi \cap \calH = \{0\}$ as the opposite would ensure that $\mrk \calH \leq \rk \varphi \leq n<2n-2$ by Lemma \ref{basiclemma}.
This shows that $\rk \varphi=n$ and that we may find a non-zero operator $g \in \calS$ such that
$\calS=\calH \oplus \Ker \varphi$ and $\Ker \varphi=\K g$.
As in the proof of Theorem \ref{minranktheo1}, the assumption $\rk g=n$ would lead to $\mrk \calH \leq n<2n-2$.
Therefore, $\rk g<n$. The line of reasoning from the proof of Theorem \ref{minranktheo1} yields $\mrk \calH \leq \rk g+(n-1)$,
whence $\rk g=n-1$.

Now, using Lemma \ref{basiclemma}, we can find a basis $\calC$ of $V$ in which the first $(n-1)$ vectors span $\im g$, and
we can choose a basis $\calB$ of $\calS$ such that $\varphi$ is represented in $\calB$ and $\calC$ by
$\begin{bmatrix}
I_n & [0]_{n \times 1} \\
[0]_{(m-n) \times n} & [0]_{(m-n) \times 1}
\end{bmatrix}$, so that the last vector of $\calB$ is $g$ and the first $n$ vectors span $\calH$.
We denote by $\calM$ the vector space of matrices representing the operators of $\widehat{\calS}$ in the bases $\calB$ and $\calC$.
Then:
\begin{enumerate}[(a)]
\item $\calM$ is reduced.
\item Every $M \in \calM$ splits up as
$$M=\begin{bmatrix}
[?]_{(n-1) \times n} & [?]_{(n-1) \times 1} \\
[?]_{(m-n+1) \times n} & [0]_{(m-n+1) \times 1}
\end{bmatrix}.$$
\item $\calM$ contains
$\begin{bmatrix}
I_n & [0]_{n \times 1} \\
[0]_{(m-n) \times n} & [0]_{(m-n) \times 1}
\end{bmatrix}$.
\item For every $X \in \K^{n-1}$, there exists a matrix in $\calM$ with
last column $\begin{bmatrix}
X \\
[0]_{(m-n+1) \times 1}
\end{bmatrix}$.
\end{enumerate}
Modifying the last $m-n$ vectors of $\calC$ further, we may assume that we have found an integer $r \in \lcro 0,m-n\rcro$
such that every matrix $M$ in $\calM$
splits up as
$$M=\begin{bmatrix}
[?]_{n \times (n-1)} & [?]_{n \times 1} & [?]_{n \times 1} \\
J(M) & [?]_{r \times 1} & [0]_{r \times 1} \\
[0]_{(m-n-r) \times 1} & H(M) & [0]_{(m-n-r) \times 1}
\end{bmatrix}$$
with $J(M) \in \Mat_{r,n-1}(\K)$ and $H(M) \in \Mat_{m-n-r,1}(\K)$,
and so that $\K^r$ is the sum of all column spaces of the matrices $J(M)$, for $M \in \calM$.

\begin{claim}\label{claimHiszero}
One has $H(M)=0$ for all $M \in \calM$, and $r=m-n$.
\end{claim}

\begin{proof}
Assume that $H(\calM) \neq \{0\}$.
For $M \in \calM$, denote by $S(M)$ the matrix obtained by deleting the last $m-n-r$ rows and the $n$-th column.
Using Lemma \ref{decompositionlemma}, we see that $\urk S(\calM)+\urk H(\calM) \leq \urk \calM=n$.
As $H(\calM) \neq 0$, we deduce that $\urk S(\calM) \leq n-1$. If we write $\calB=(f_1,\dots,f_n,g)$, it follows that
$\Vect(f_1,\dots,f_{n-1},g)$ is LLD, contradicting the fact that $\calS$ is a minimal LLD space.
Thus, $H(\calM)=\{0\}$, which leads to $m-n-r=0$ since $\calM$ is reduced.
\end{proof}

Next, we split every $M$ of $\calM$ up as
$$M=\begin{bmatrix}
[?]_{(n-1) \times (n-1)} & [?]_{(n-1) \times 1} & C(M) \\
R(M) & ? & 0 \\
J(M) & [?]_{(m-n)\times 1} & [0]_{(m-n) \times 1}
\end{bmatrix},$$
with $R(M) \in \Mat_{1,n-1}(\K)$ and $C(M) \in \K^{n-1}$.

\begin{claim}
For every $M \in \calM$, equality $C(M)=0$ implies $J(M)=0$.
\end{claim}

\begin{proof}
Let $M \in \calM$ be such that $C(M)=0$.
Let $X \in \K^{n-1}$. Then, some $N \in \calM$ has last column $\begin{bmatrix}
X \\
[0]_{(m-n+1) \times 1}
\end{bmatrix}$. Thus, for all $\lambda \in \K$, the matrix
$N+\lambda M$ has last column $\begin{bmatrix}
X \\
[0]_{(m-n+1) \times 1}
\end{bmatrix}$, and hence the Flanders-Atkinson lemma yields $(J(N)+\lambda J(M))X=0$.
Varying $\lambda$ yields $J(M)X=0$. As this holds for all $X \in \K^{n-1}$, the claimed statement ensues.
\end{proof}

Thus, one obtains a linear map $K : \K^{n-1} \rightarrow \Mat_{m-n,n-1}(\K)$ such that
every $M \in \calM$ splits up as
$$M=\begin{bmatrix}
[?]_{(n-1) \times (n-1)} & [?]_{(n-1) \times 1} & C(M) \\
R(M) & ? & 0 \\
K(C(M)) & [?]_{(m-n)\times 1} & [0]_{(m-n) \times 1}
\end{bmatrix}.$$

Using the Flanders-Atkinson lemma, we note further that
$$\forall X \in \K^{n-1}, \; K(X)X=0.$$

\vskip 3mm
Next, we show that $R(M)$ is also solely dependent on $C(M)$.

\begin{claim}\label{claimRdependsonC}
For every $M \in \calM$, equality $C(M)=0$ implies $R(M)=0$.
\end{claim}

\begin{proof}
Let $N_0 \in \calM$ be such that $C(N_0)=0$. Set $L_0:=R(N_0)$.
We shall use a line of reasoning that is largely similar to that of our proof of Theorem \ref{refinedMStheogeneralized}.
A non-zero vector $X \in \K^{n+1}$ is called $\calM$-good when $X$ belongs to the kernel of some
rank $n$ matrix of $\calM$. For such a vector, we know from Lemma \ref{basiclemma} that $\dim \calM X \leq n$.

In a basis $\bfD$ of $\calM$ and the canonical basis $\bfB_c$ of $\K^m$, we consider the
space of all matrices representing the operators $\check{X} : M \in \calM \mapsto MX \in \K^m$, for $X \in \K^{n+1}$.
As obviously $\urk \calS \geq \mrk \calH>n$, we may find respective subsets
$I$ and $J$ of $\lcro 1,m\rcro$ and $\lcro 1,\dim \calM\rcro$, both with cardinality $n+1$,
such that the function $X \in \K^{n+1} \mapsto \det \Mat_{\bfD,\bfB_c}(\check{X})(I \mid J)$ does not vanish everywhere
on $\K^{n+1}$ (with the notation from the proof of Theorem \ref{refinedMStheogeneralized}).
 We denote by $G \in \K[\mathbf{x_1},\dots,\mathbf{x_{n+1}}]$ the associated $(n+1)$-homogeneous polynomial, and note
that $G$ is non-zero but vanishes at every $\calM$-good vector.

As in the proof of Theorem \ref{refinedMStheogeneralized}, we see that $G$ is unmodified by extending the ground field since $\# \K\geq n+1$.
Taking the quotient field $\L:=\K((s,t))$ of the ring of formal power series in two independent
variables $s$ and $t$, we deduce that $G$ vanishes at every $\calM_\L$-good
vector of $\L^{n+1}$.

Writing $N_0=\begin{bmatrix}
T & [0]_{n \times 1} \\
[?]_{(m-n) \times n} & [0]_{(m-n) \times 1}
\end{bmatrix}$, with $T \in \Mat_n(\K)$, we know that $\calM_\L$ contains a matrix of the form
$$J=\begin{bmatrix}
I_n-s T & [0]_{n \times 1} \\
[?]_{(m-n) \times n} & [0]_{(m-n) \times 1}
\end{bmatrix},$$
and $I_n-sT$ belongs to $\GL_n(\L)$.

Let $X \in \K^{n-1}$, and set $\widetilde{X}:=\begin{bmatrix}
X \\
0
\end{bmatrix} \in \K^n$.
We know that $\calM$ contains a matrix of the form
$$M=\begin{bmatrix}
A & \widetilde{X} \\
[?] & [?]
\end{bmatrix}, \quad \text{where $A \in \Mat_n(\K)$.}$$
Thus, $\calM_\L$ contains a matrix of the form
$$\begin{bmatrix}
(I_n-sT)-tA & -t\widetilde{X} \\
[?] & [?]
\end{bmatrix}$$
and, as $(I_n-sT)-tA$ is obviously invertible, one deduces that the vector
$$\begin{bmatrix}
t\bigl((I_n-sT)-tA\bigr)^{-1}\widetilde{X} \\
1
\end{bmatrix}$$
is $\calM_\L$-good.
Thus,
\begin{equation}\label{powerseries2}
G\Bigl(t\bigl((I_n-sT)-tA\bigr)^{-1}\widetilde{X}\, ,\,1\Bigr)=0.
\end{equation}
As we have assumed that $G \neq 0$, and on the other hand $G(0,\dots,0,1)=0$ since the last vector of the canonical basis of $\K^{n+1}$
is $\calM$-good, we may find some $d \in \lcro 1,n+1\rcro$ such that
$$G=\underset{k=d}{\overset{n+1}{\sum}}G_k(\mathbf{x_1},\dots,\mathbf{x_n})\,(\mathbf{x_{n+1}})^{n+1-k},$$
where each $G_k$ is a $k$-homogeneous polynomial in $n$ variables, and
$G_d \neq 0$. Then, by expanding modulo $t^{d+1}$ as in the proof of Theorem \ref{refinedMStheogeneralized},
one finds
$$G_d((I_n-sT)^{-1}\widetilde{X})=0.$$
Again, we may expand
$$G_d=\sum_{k=c}^d H_{d-k}(\mathbf{x_1},\dots,\mathbf{x_{n-1}}) (\mathbf{x_n})^k,$$
where each $H_i$ is an $i$-homogeneous polynomial, $c \in \lcro 0,d\rcro$ and $H_{d-c} \neq 0$.
Remembering that $T=\begin{bmatrix}
[?]_{(n-1) \times (n-1)} & [?]_{(n-1) \times 1} \\
L_0 & ?
\end{bmatrix}$, we see that $(I_n-s T)^{-1} \widetilde{X}=\begin{bmatrix}
Y \\
p
\end{bmatrix}$, where $Y=X$ mod.\ $s$ and $p=s\,L_0X$ mod.\ $s^2$.
It follows that
$$G_d\bigl((I_n-sT)^{-1}\widetilde{X}\bigr)= H_{d-c}(X)(L_0X)^c s^c\quad \text{mod. $s^{c+1}$.}$$
We conclude that
$$\forall X \in \K^{n-1}, \; H_{d-c}(X)(L_0X)^c=0.$$
As $X \mapsto H_{d-c}(X)(L_0X)^c$ is a $d$-homogeneous polynomial function (and $\# \K \geq n+1 \geq d$)
whereas $H_{d-c}\neq 0$, we find that $(L_0X)^c=0$ for all $X \in \K^{n-1}$. This yields $L_0=0$, as claimed.
\end{proof}

From now on, it will be convenient to use a parametrization of $\calM$.
First of all, we may find linear maps $A : \K^{n-1} \rightarrow \Mat_{n-1}(\K)$,
$L : \K^{n-1} \rightarrow \Mat_{1,n-1}(\K)$, $C_1 : \K^{n-1} \rightarrow \K^{m-n}$ such that, for every
$X \in \K^{n-1}$, the space $\calM$ contains a matrix of the form
$$M_X=\begin{bmatrix}
A(X) & [?]_{(n-1) \times 1} & X \\
L(X) & ? & 0 \\
K(X) & C_1(X) & [0]_{(m-n) \times 1}
\end{bmatrix}.$$
Setting $\calN:=\Ker C$ and using Claim \ref{claimRdependsonC}, we may also find linear maps $B : \calN \rightarrow \Mat_{n-1}(\K)$,
and $C_2 : \calN \rightarrow \K^{m-n}$ such that every $N \in \calN$ splits up as
$$N=\begin{bmatrix}
B(N) & [?]_{(n-1) \times 1} & [0]_{(n-1) \times 1} \\
[0]_{1 \times (n-1)} & ? & 0 \\
[0]_{(m-n) \times (n-1)} & C_2(N) & [0]_{(m-n) \times 1}
\end{bmatrix}.$$
Note that
$$\calM=\bigl\{M_X \mid X \in \K^{n-1}\bigr\} \oplus \calN.$$

\begin{claim}
One has $B(\calN)X =\K^{n-1}$
for every non-zero vector $X \in \K^{n-1}$.
\end{claim}

\begin{proof}
Let $X \in \K^{n-1} \setminus \{0\}$ seen as a vector of $\K^{n+1}$ through the canonical identification map
$\K^{n-1} \overset{\simeq}{\rightarrow} \K^{n-1} \times \{0\}$. By the rank theorem, one finds
$$2n-2 \leq \dim \calM X \leq \dim \calN X+\dim C(\calM)=\dim \calN X+(n-1).$$
Therefore, $\dim \calN X \geq n-1$. This proves our claim, judging from the form of the matrices in $\calN$.
\end{proof}

Now, fix $X \in \K^{n-1}$ and $N \in \calN$. Let $\lambda \in \K$. Applying the Flanders-Atkinson lemma to
$M_X+\lambda N$ for $k=1$, one finds:
$$K(X)\bigl(A(X)+\lambda B(N)\bigr)X+L(X)X\,\bigl(C_1(X)+\lambda C_2(N)\bigr)=0.$$
As this holds for all $\lambda \in \K$, one finds
\begin{equation}\label{bigidentity1}
\forall (X,N) \in \K^{n-1} \times \calN, \; K(X)\,B(N)\,X+L(X)X\,C_2(N)=0.
\end{equation}
and
\begin{equation}\label{bigidentity2}
\forall X \in \K^{n-1}, \; K(X)\,A(X)\,X+L(X)X\,C_1(X)=0.
\end{equation}

\begin{claim}\label{ontoclaim}
The map $C_2 : \calN \rightarrow \K^{m-n}$ is onto.
\end{claim}

\begin{proof}
Assume on the contrary that $C_2$ is not onto.
Note that none of the previous assumptions is changed by left-multiplying $\calM$
with a matrix of the form $I_n \oplus P$ for some $P \in \GL_{m-n}(\K)$.
Therefore, we lose no generality in assuming that $C_2$ maps $\calN$ into $\K^{m-n-1} \times \{0\}$.
Extracting the last row in identity \eqref{bigidentity1} yields
$$\forall (X,N) \in \K^{n-1} \times \calN, \quad K_{m-n}(X)B(N)X=0,$$
where $K_{m-n}(X)$ denotes the last row of $K(X)$.
However, fixing a non-zero vector $X \in \K^{n-1}$, one deduces from $B(\calN)X=\K^{n-1}$ that
$K_{m-n}(X)=0$. Thus, the last row of $K(X)$ is zero for all $X \in \K^{n-1}$, contradicting the
assumption that $\K^{m-n}$ be the sum of all the column spaces of the matrices in $J(\calM)$.
\end{proof}

Using the fact that $\calM$ contains $\begin{bmatrix}
I_n & [0]_{n \times 1} \\
[0]_{(m-n) \times n} & [0]_{(m-n) \times 1}
\end{bmatrix}$, it ensues that, for all $Y \in \K^{m-n}$, the space
$\calM$ contains a matrix of the form
$$N_Y=\begin{bmatrix}
[?]_{(n-1) \times (n-1)} & [?]_{(n-1) \times 1} & [0]_{(n-1) \times 1} \\
[0]_{(m-n+1) \times (n-1)} & Y & [0]_{(m-n+1) \times 1}
\end{bmatrix}.$$

\begin{claim}\label{claimm=2n-2}
One has $m=2n-2$.
\end{claim}

\begin{proof}
For every $M \in \calM$, let us write
$$M=\begin{bmatrix}
[?]_{(n-1) \times n} & [?]_{(n-1) \times 1} \\
D(M) & [0]_{(m-n+1) \times 1}
\end{bmatrix} \quad \text{with $D(M) \in \Mat_{m-n+1,n}(\K)$.}$$
On the one hand, Lemma \ref{decompositionlemma} yields
$\urk D(\calM) \leq n-1$. On the other hand, $\dim D(\calM)X \geq n-1$ for all non-zero vectors $X \in \K^n$, since
$\rk f \geq 2n-2$ for all non-zero operators $f \in \calH$.
Theorem \ref{refinedMStheogeneralized} applied to the LLD operator space $\bigl\{M \mapsto D(M)X \mid X \in \K^n\}$
shows that $\dim D(\calM)X=n-1$ for all non-zero vectors $X \in \K^n$.

However, the matrices of type $N_Y$ show that the vector $X_0:=\begin{bmatrix}
0 & \cdots & 0 & 1
\end{bmatrix}^T\in \K^n$ satisfies $\dim D(\calM) X_0=m-n+1$. Therefore, $m-n+1=n-1$ as claimed.
\end{proof}

Subtracting matrices of type $N_Y$, we may now assume that, for all $X \in \K^n$,
$$M_X=
\begin{bmatrix}
A(X) & C_3(X) & X \\
L(X) & 0 & 0 \\
K(X) & [0]_{(n-1) \times 1} & [0]_{(n-1) \times 1}
\end{bmatrix} \quad \text{with $C_3(X) \in \K^{n-1}$,}$$
and $C_3$ is an endomorphism of $\K^{n-1}$.

Now, for $X \in \K^{n-1}$, we consider the matrix
$$E(X):=\begin{bmatrix}
L(X) \\
K(X)
\end{bmatrix}\in \Mat_{n-1}(\K).$$

\begin{claim}\label{EXinvertible}
For every non-zero vector $X$ of $\K^{n-1}$, the matrix $E(X)$ is invertible
and the kernel of $K(X)$ is $\K X$.
\end{claim}

\begin{proof}
Let $X \in \K^{n-1}$ be a non-zero vector. Assume that $E(X)$ is singular, and
choose a non-zero vector $Y \in \K^{n-1}$ such that $E(X)Y=0$.
The rank theorem yields
$$\dim \calM \begin{bmatrix}
Y \\
0 \\
0
\end{bmatrix}
 \leq (n-1)+\dim E(\K^{n-1})Y<(n-1)+(n-1),$$
which contradicts the assumption that $\mrk \calH=2n-2$.
Thus, $E(X)$ is non-singular and it follows that $\rk K(X)=n-2$.
As $K(X)X=0$ (by the Flanders-Atkinson lemma) and $X$ is non-zero, one concludes that $X$ spans the kernel of $K(X)$.
\end{proof}

Now, for $N \in \calN$, we write
$$N=\begin{bmatrix}
B(N) & C_4(N) & [0]_{(n-1) \times 1} \\
[0]_{(n-1) \times (n-1)} & C_5(N) & [0]_{(n-1) \times 1}
\end{bmatrix}
\quad \text{with $(C_4(N),C_5(N)) \in (\K^{n-1})^2$.}$$

Fix $N \in \calN$ together with a non-zero vector $X \in \K^{n-1}$.
Then, we have
$$N+M_X=\begin{bmatrix}
A(X)+B(N) & C_3(X)+C_4(N) & X \\
E(X) & C_5(N) & [0]_{(n-1) \times 1}
\end{bmatrix},$$
and we choose a non-zero vector $\begin{bmatrix}
Z_0 \\
\lambda \\
\mu
\end{bmatrix}$ in its kernel, with $Z_0 \in \K^{n-2}$ and $(\lambda,\mu) \in \K^2$.
This leads to
$$\bigl(A(X)+B(N)\bigr)Z_0+\lambda \,C_3(X)+\lambda \,C_4(N)+\mu \,X=0 \quad \text{and} \quad
E(X)Z_0+\lambda\, C_5(N)=0.$$
The second identity yields $Z_0=-\lambda\,E(X)^{-1}C_5(N)$.
If $\lambda=0$, then we would deduce that $Z_0=0$, and the first identity above would yield $\mu X=0$, and hence $\mu=0$.
Therefore $\lambda \neq 0$, and hence the first identity yields:
\begin{equation}\label{majoridentity}
\forall X \in \K^{n-1} \setminus \{0\}, \; \forall N \in \calN, \; \bigl(A(X)+B(N)\bigr)E(X)^{-1} C_5(N)
-C_3(X)-C_4(N) \in \K X.
\end{equation}

\begin{claim}
There is a scalar $\lambda_n \in \K$ such that $C_3(X)=\lambda_n\,X$ for all $X \in \K^{n-1}$.
\end{claim}

\begin{proof}
Let $X \in \K^{n-1} \setminus \{0\}$. Applying identity \eqref{majoridentity} to $N=0$ yields
$C_3(X) \in \K X$.  The claimed result ensues since $X \mapsto C_3(X)$ is an endomorphism of $\K^{n-1}$.
\end{proof}

Thus, identity \eqref{majoridentity} yields:
\begin{equation}\label{majoridentity2}
\forall X \in \K^{n-1} \setminus \{0\}, \; \forall N \in \calN, \; \bigl(A(X)+B(N)\bigr)E(X)^{-1} C_5(N)-C_4(N) \in \K X.
\end{equation}

\begin{claim}
The map $C_5 : \calN \rightarrow \K^{n-1}$ is an isomorphism.
\end{claim}

\begin{proof}
We know that $C_5$ is linear, and we have seen right above the statement of Claim \ref{claimm=2n-2} that $C_5$ is onto.
Let us prove that $C_5$ is one-to-one.
Assume that some non-zero $N_0 \in \calN$ satisfies $C_5(N_0)=0$.
Then, identity \eqref{majoridentity2} yields $C_4(N_0) \in \K X$ for all non-zero
vectors $X \in \K^{n-1}$. Since $n \geq 3$, this yields $C_4(N_0)=0$.
As $C_2(N_0)=0$, identity \eqref{bigidentity1}
yields $\forall X \in \K^{n-1}, \; K(X)B(N_0)X=0$, and hence $\forall X \in \K^{n-1}, \; B(N_0)X \in \K X$ by Claim \ref{EXinvertible}.
It follows that $B(N_0)=\alpha I_{n-1}$ for some $\alpha \in \K$. Since $N_0 \neq 0$ and $C_4(N_0)=C_5(N_0)=0$, one must
have $\alpha \neq 0$. Combining linearly $N_0$ with $\begin{bmatrix}
I_n & [0]_{n \times 1} \\
[0]_{(n-2) \times n} & [0]_{(n-2) \times 1}
\end{bmatrix}$, we deduce that $\calM$ contains the matrix
$E_{n,n} \in \Mat_{2n-2,n+1}(\K)$ with all entries zero save the one at the $(n,n)$-spot, which equals one.

Given a matrix $M \in \calM$, we denote by $M'$ the matrix of $\Mat_{2n-3,n}(\K)$ obtained by deleting the $n$-th row and $n$-th column of
$M$, thus giving rise to a linear subspace $\calM'$ of $\Mat_{2n-3,n}(\K)$.
We contend that $\urk \calM'\leq n-1$. Assume on the contrary that $\calM$ contains a matrix $M$
such that $\rk M'=n$. As $\rk M \leq n$, one deduces that the $n$-th column of $M$ is a linear combination
of the other columns. Applying the same principle to $M+E_{n,n}$, one deduces that the $n$-th column $C_0$ of
$E_{n,n}$ is a linear combination of the columns of $M$, and hence the matrix $M_1$ obtained from $M$ by
replacing its $n$-th column by $C_0$ has also rank less than $n+1$. However, one sees that
$\rk M_1=1+\rk M'=n+1$, a contradiction.

Now, not only do we have $\urk \calM'\leq n-1$, but the assumptions on $\calM$ also show that
$\calM'$ contains $\begin{bmatrix}
I_{n-1} & [0]_{(n-1) \times 1} \\
[0]_{(n-2) \times (n-1)} & [0]_{(n-2) \times 1}
\end{bmatrix}$ and that, for every $X \in \K^{n-1}$, it contains a matrix of the form
$\begin{bmatrix}
[?]_{(n-1) \times (n-1)} & X \\
[?]_{(n-2) \times (n-1)} & [0]_{(n-2) \times 1}
\end{bmatrix}$. Thus, Theorem \ref{refinedMStheogeneralized} shows that $\dim \calM' Y\leq n-1$ for all $Y \in \K^n$.
In particular, denoting by $e_1$ the first vector of the canonical basis of $\K^{n+1}$, one deduces that
$$\dim \calM e_1 \leq 1+(n-1) = n<2n-2,$$
contradicting our assumptions. Thus, $C_5$ is one-to-one.
\end{proof}

In particular, the rank theorem yields
$$\dim \calM=\dim \calN+(n-1)=2(n-1).$$
Moreover, we find a (unique) matrix $D \in \Mat_{n-1}(\K)$ such that $C_4(N)=D \,C_5(N)$ for all $N \in \calN$.
As $\calN$ contains $\begin{bmatrix}
I_n & [0]_{n \times 1} \\
[0]_{(n-2) \times n} & [0]_{(n-2) \times 1}
\end{bmatrix}$, one deduces that the first column of $D$ is zero.
Setting
$P:=\begin{bmatrix}
I_{n-1} & -D \\
[0]_{(n-1) \times (n-1)} & I_{n-1}
\end{bmatrix}$, we have
$P=\begin{bmatrix}
I_n & [?]_{n \times (n-2)} \\
[0]_{(n-2) \times n} & I_{n-2}
\end{bmatrix}$, and hence none of the previous assumptions is lost in replacing $\calM$ with $P\calM$
(as we have only modified the last $n-2$ vectors of the chosen basis of $V$
by adding to each one of them a well-chosen linear combination of the first $n$ ones).
In the new space $\calM$, we have $C_4(N)=0$ for all $N \in \calN$.

In this reduced situation, we may find a linear map $F : \K^{n-1} \rightarrow \Mat_{n-1}(\K)$ such that
$B(N)=F(C_5(N))$ for all $N \in \calN$. Thus, $\calN$ is the set of all matrices
$$\begin{bmatrix}
F(Y) & [0]_{(n-1) \times 1} & [0]_{(n-1) \times 1} \\
[0]_{(n-1) \times (n-1)} & Y & [0]_{(n-1) \times 1}
\end{bmatrix} \quad \text{with $Y \in \K^{n-1}$.}$$

We have already shown that $\dim F(\K^{n-1})X=\dim B(\calN) X=n-1$ for all non-zero vectors $X \in \K^{n-1}$.
On the other hand, we have proved that $\dim \calN=n-1$.
With the same line of reasoning as in the proof of Claim \ref{EXinvertible}, this yields:

\begin{claim}\label{FYinvertible}
For every non-zero vector $Y \in \K^{n-1}$, the matrix $F(Y)$ is invertible.
\end{claim}

Identity \eqref{majoridentity2} now reads:
$$\forall X \in \K^{n-1} \setminus \{0\}, \; \forall Y \in \K^{n-1}, \;
A(X)E(X)^{-1}Y+F(Y)E(X)^{-1}Y \in \K X.$$
Fix $X \in \K^{n-1} \setminus \{0\}$. As $K(X)X=0$, we deduce that
$$\forall Y \in \K^{n-1}, \; K(X) A(X)E(X)^{-1}Y+K(X)F(Y)E(X)^{-1}Y=0.$$
Noting that $Y \mapsto K(X) A(X)E(X)^{-1}Y$ is linear and $Y \mapsto K(X)F(Y)E(X)^{-1}Y$ is quadratic
(and $\# \K > 2$),  this yields the two identities
\begin{equation}\label{majoridentity3}
\forall Y \in \K^{n-1}, \;  K(X) A(X)E(X)^{-1}Y=0
\end{equation}
and
\begin{equation}\label{majoridentity4}
\forall Y \in \K^{n-1}, \;  K(X)F(Y)E(X)^{-1}Y=0.
\end{equation}
First of all, \eqref{majoridentity4} can be rephrased as
\begin{equation}\label{majoridentity5}
\forall Y \in \K^{n-1}, \;  F(Y)E(X)^{-1}Y \in \K X.
\end{equation}
On the other hand, identity \eqref{majoridentity3} yields $\im(A(X)E(X)^{-1}) \subset \K X$.
Since $E(X)^{-1}$ is invertible, this shows that $\im A(X) \subset \K X$.
As this holds for all non-zero vectors $X \in \K^{n-1}$, and as $X\mapsto A(X)$ is linear, this yields scalars $\lambda_1,\dots,\lambda_{n-1}$
such that
$$\forall X \in \K^{n-1}, \; A(X)=\begin{bmatrix}
\lambda_1\,X & \cdots & \lambda_{n-1}\,X
\end{bmatrix}.$$
We are almost ready to conclude. Now, we can forget the assumption that $\dim \calM Z=2n-2$ for all non-zero vectors $Z \in \K^n \times \{0\}$
and focus on finding a matrix space which is equivalent to $\calM$ and represents the dual operator space of the twisted
operator space associated with an LDB division algebra.
Recall that,
for all $X \in \K^{n-1}$, we have
$$M_X=\begin{bmatrix}
A(X) & \lambda_n X & X \\
K(X) & [0]_{(n-1) \times 1} & [0]_{(n-1) \times 1}
\end{bmatrix}.$$
Remember also that $\calM=\{M_X \mid X \in \K^{n-1}\} \oplus \calN$.

Using the column operations $C_i \leftarrow C_i-\lambda_i C_{n+1}$ for $i$ from $1$ to $n$,
followed by the column operation $C_n \leftrightarrow C_{n+1}$,
one deduces that $\calM$ is equivalent to the space $\calM_1$ of all matrices of the form
$$\begin{bmatrix}
F(Y) & X & [0]_{(n-1) \times 1}  \\
E(X) & [0]_{(n-1) \times 1} & Y
\end{bmatrix} \quad \text{with $(X,Y) \in (\K^{n-1})^2$.}$$

By Claims \ref{EXinvertible} and \ref{FYinvertible},
we see that the rules
$$\forall (X,Y) \in (\K^{n-1})^2, \quad X \star Y:=F(Y)X \quad \text{and} \quad X \bullet Y:=E(Y)X$$
define two regular bilinear pairings $\star$ and $\bullet$ from $\K^{n-1} \times \K^{n-1}$ to $\K^{n-1}$.
We shall conclude by proving that $\calA:=(\K^{n-1},\star,\bullet)$ is an LDB division algebra and that $\calM_1$ represents the \emph{dual} operator
space of $\calT_\calA$ in well-chosen bases.

Given $(X,Y)\in (\K^{n-1})^2$, with $X$ non-zero, identity
\eqref{majoridentity5} applied to the pair $(Y,E(Y)X)$ reads
$$F\bigl(E(Y)X\bigr)\,E(Y)^{-1}E(Y)X \in \K Y,$$
and hence
$$X \star (X \bullet Y) \in \K Y.$$
With $X=0$, the above statement is obvious, and hence $\bullet$ is a quasi-left-inversion of~$\star$.
Thus, $\calA$ is an LDB division algebra and the dual operator space $\widehat{\calT_\calA}$ is equivalent to the space
$\calV$ of all linear maps
$$f_{Y,Z} : (X,\lambda,\mu)\in \K^{n+1} \mapsto (F(Z)X+\lambda Y\,,\, E(Y)X+\mu Z) \quad \text{with $(Y,Z) \in \K^{n-1} \times \K^{n-1}$.}$$
Given $(Y,Z) \in \K^{n-1} \times \K^{n-1}$, one checks that the matrix of $f_{Y,Z}$ in the canonical bases of
$\K^{n+1}$ and $\K^{n-1} \times \K^{n-1}$ is precisely
$$\begin{bmatrix}
F(Z) &  Y & [0]_{(n-1) \times 1}  \\
E(Y) & [0]_{(n-1) \times 1} & Z
\end{bmatrix}.$$
Therefore, $\calM_1$ represents $\calV$, and hence it also represents $\widehat{\calT_\calA}$.
As $\calM_1$ represents $\widehat{\calS}$, one concludes that $\calS$ is equivalent to $\calT_\calA$.
This completes the proof of Theorem \ref{theocaslimite}.

\end{document}